\documentclass{amsart}
\usepackage{amssymb}
\usepackage{amsfonts}
\usepackage{geometry}

\newtheorem{theorem}{Theorem}
\theoremstyle{plain}

\newtheorem{corollary}[theorem]{Corollary}

\newtheorem{definition}[theorem]{Definition}
\newtheorem{example}[theorem]{Example}

\newtheorem{lemma}[theorem]{Lemma}

\newtheorem{problem}[theorem]{Problem}

\newtheorem{remark}[theorem]{Remark}

\numberwithin{equation}{section}
\input{tcilatex}
\begin{document}
\title[Control of $BICT$]{Control of the Bilinear Indicator Cube Testing
property}

\begin{abstract}
We show that the $\alpha $-fractional Bilinear Indicator/Cube Testing
Constant%
\begin{equation*}
\mathcal{BICT}_{T^{\alpha }}\left( \sigma ,\omega \right) \equiv \sup_{Q\in 
\mathcal{P}^{n}}\sup_{E,F\subset Q}\frac{1}{\sqrt{\left\vert Q\right\vert
_{\sigma }\left\vert Q\right\vert _{\omega }}}\left\vert \int_{F}T_{\sigma
}^{\alpha }\left( \mathbf{1}_{E}\right) \omega \right\vert ,
\end{equation*}%
defined for any $\alpha $-fractional singular integral $T^{\alpha }$ on $%
\mathbb{R}^{n}$ with $0<\alpha <n$, is controlled by the classical $\alpha $%
-fractional Muckenhoupt constant $A_{2}^{\alpha }\left( \sigma ,\omega
\right) $, provided the product measure $\sigma \times \omega $ is
diagonally reverse doubling (in particular if it is reverse doubling) with
exponent exceeding $2\left( n-\alpha \right) $.

Moreover, this control is sharp within the class of diagonally reverse
doubling product measures. In fact, every product measure $\mu \times \mu $,
where $\mu $ is an Ahlfors-David regular measure $\mu $ with exponent $%
n-\alpha $, has diagonal exponent $2\left( n-\alpha \right) $ and satisfies $%
A_{2}^{\alpha }\left( \mu ,\mu \right) <\infty $ and $\mathcal{BICT}%
_{I^{\alpha }}\left( \mu ,\mu \right) =\infty $, which in paricular has
implications for the $L^{2}$ trace inequality of the fractional integral $I^{%
\frac{\alpha }{2}}$ on domains with fractional boundary.

When combined with the main results in \texttt{arXiv:1906.05602}, \texttt{%
1907.07571}\ and \texttt{1907.10734}, the above control of $\mathcal{BICT}%
_{T^{\alpha }}$ for $\alpha >0$ yields a two weight $T1$ theorem for
doubling weights with appropriate diagonal reverse doubling, i.e. the norm
inequality for $T^{\alpha }$ is controlled by cube testing constants and the 
$\alpha $-fractional one-tailed Muckenhoupt constants $\mathcal{A}%
_{2}^{\alpha }$ (without any energy assumptions), and also yields a
corresponding cancellation condition theorem for the kernel of $T^{\alpha }$%
, both of which hold for arbitrary $\alpha $-fractional Calder\'{o}n-Zygmund
operators $T^{\alpha }$.

We do not know if the analogous result for $\mathcal{BICT}_{H}\left( \sigma
,\omega \right) $ holds for the Hilbert transform $H$ in case $\alpha =0$,
but we show that $\mathcal{BICT}_{H^{\limfunc{dy}}}\left( \sigma ,\omega
\right) $ is \textbf{not} controlled by the Muckenhoupt condition $\mathcal{A%
}_{2}^{\alpha }\left( \omega ,\sigma \right) $ for the dyadic Hilbert
transform $H^{\limfunc{dy}}$ and doubling weights $\sigma ,\omega $.
\end{abstract}

\author[E.T. Sawyer]{Eric T. Sawyer}
\address{ Department of Mathematics \& Statistics, McMaster University, 1280
Main Street West, Hamilton, Ontario, Canada L8S 4K1 }
\thanks{E. Sawyer has been partially supported by an NSERC grant, and a
startup\ grant from McMaster University for the McKay professorship.}
\author[I. Uriarte-Tuero]{Ignacio Uriarte-Tuero}
\address{ Department of Mathematics \\
Michigan State University \\
East Lansing MI }
\email{ignacio@math.msu.edu}
\thanks{I. Uriarte-Tuero has been partially supported by grant
MTM2015-65792-P (MINECO, Spain), and by a Simons Foundation Collaboration
Grant for Mathematicians, Award Number: 637221.}
\email{sawyer@mcmaster.ca}
\maketitle
\tableofcontents

\section{Introduction}

We give precise statements of our main results in Subsection \ref{main}
below, but first we recall the definitions of doubling, reverse doubling,
Muckenhoupt conditions and Poisson integrals; then the notion of weighted
norm inequality for a standard singular integral, and the associated testing
conditions; and finally the Bilinear Indicator/Cube Testing theorem from 
\cite{Saw2}, \cite{Saw3} and \cite{Saw4}.

\subsection{Definitions}

Denote by $\mathcal{P}^{n}$ the collection of cubes in $\mathbb{R}^{n}$
having sides parallel to the coordinate axes. A positive locally finite
Borel measure $\mu $ on $\mathbb{R}^{n}$ is said to satisfy the\emph{\
doubling condition} if there is a pair of constants $\left( \beta ,\gamma
\right) \in \left( 0,1\right) ^{2}$, called doubling parameters, such that
with $\left\vert Q\right\vert _{\mu }=\mu \left( Q\right) $, 
\begin{equation}
\left\vert \beta Q\right\vert _{\mu }\geq \gamma \left\vert Q\right\vert
_{\mu }\ ,\ \ \ \ \ \text{for all cubes }Q\in \mathcal{P}^{n},
\label{def doub}
\end{equation}%
and the \emph{reverse doubling condition} if there is a pair of constants $%
\left( \beta ,\gamma \right) \in \left( 0,1\right) ^{2}$, called reverse
doubling parameters, such that%
\begin{equation}
\left\vert \beta Q\right\vert _{\mu }\leq \gamma \left\vert Q\right\vert
_{\mu }\ ,\ \ \ \ \ \text{for all cubes }Q\in \mathcal{P}^{n}.
\label{def rev doub'}
\end{equation}%
Note that the inequality in (\ref{def rev doub'}) has been reversed from
that in the definition of the doubling condition in (\ref{def doub}).

A familiar equivalent reformulation of (\ref{def doub}) is that there is a
positive constant $C_{\limfunc{doub}}$, called the \emph{doubling constant},
and a positive constant $C$, such that $\left\vert 2Q\right\vert _{\mu }\leq
C_{\limfunc{doub}}\left\vert Q\right\vert _{\mu }$ for all cubes $Q\in 
\mathcal{P}^{n}$. More important for us is yet another characterization that
follows by iterating (\ref{def doub}): $\mu $ is doubling if and only if
there exists a\ positive constant $\theta _{\mu }^{\limfunc{doub}}$, called
a \emph{doubling exponent}, such that 
\begin{equation*}
\sup_{Q\in \mathcal{P}^{n}}\frac{\left\vert tQ\right\vert _{\mu }}{%
\left\vert Q\right\vert _{\mu }}\leq t^{\theta _{\mu }^{\limfunc{doub}}},\ \
\ \ \ \text{for all sufficiently large }t<\infty .
\end{equation*}%
Similarly there is the analogous reformulation of (\ref{def rev doub'}): $%
\mu $ is reverse doubling if and only if there exists a positive constant $%
\theta _{\mu }^{\limfunc{rev}}$, called a \emph{reverse doubling exponent},
and a positive constant $C$, such that%
\begin{equation*}
\sup_{Q\in \mathcal{P}^{n}}\frac{\left\vert sQ\right\vert _{\mu }}{%
\left\vert Q\right\vert _{\mu }}\leq s^{\theta _{\mu }^{\limfunc{rev}}},\ \
\ \ \ \text{for all sufficiently small }s>0.
\end{equation*}%
A doubling exponent $\theta _{\mu }^{\limfunc{doub}}$ of a doubling measure $%
\mu $ is necessarily large, namely $\theta _{\mu }^{\limfunc{doub}}\geq n$,
and a reverse doubling exponent $\theta _{\mu }^{\limfunc{rev}}$ of a
reverse doubling measure $\mu $ is necessarily small, namely $\theta _{\mu
}^{\limfunc{rev}}\leq n$, with Lebesgue measure satisfying the extreme case $%
\theta _{dx}^{\limfunc{rev}}=n=\theta _{dx}^{\limfunc{doub}}$. Indeed, with $%
\Omega _{N}\equiv \left\{ \alpha \in \mathbb{N}^{n}:0\leq \alpha _{i}\leq
N-1\right\} $, we have for $k$ large,%
\begin{equation*}
3^{kn}\left\vert 3^{k}Q\right\vert _{\mu }\leq \sum_{\alpha \in \Omega
_{3^{k}}}\left\vert 3^{k+1}\left( Q+\ell \left( Q\right) \alpha \right)
\right\vert _{\mu }\leq \sum_{\alpha \in \Omega _{3^{k}}}3^{\left(
k+1\right) \theta _{\mu }^{\limfunc{doub}}}\left\vert Q+\ell \left( Q\right)
\alpha \right\vert _{\mu }\leq 3^{\left( k+1\right) \theta _{\mu }^{\limfunc{%
doub}}}\left\vert 3^{k}Q\right\vert _{\mu }\ ,
\end{equation*}%
which implies $\theta _{\mu }^{\limfunc{doub}}\geq n$. Similarly $\theta
_{\mu }^{\limfunc{rev}}\leq n$.

Finally it is well known that doubling implies reverse doubling. Indeed,
assuming $t\geq 5$ in the definition of $\theta _{\mu }^{\limfunc{doub}}$,
we obtain for any cube $Q$ in a dyadic grid $\mathcal{D}$,%
\begin{eqnarray*}
\left\vert 3Q\setminus Q\right\vert _{\mu } &=&\sum_{I\in \mathcal{D}:\
I\subset 3Q\setminus Q,\ell \left( I\right) =\ell \left( Q\right)
}\left\vert I\right\vert _{\mu }\geq \sum_{I\in \mathcal{D}:\ I\subset
3Q\setminus Q,\ell \left( I\right) =\ell \left( Q\right) }5^{-\theta _{\mu
}^{\limfunc{doub}}}\left\vert 5I\right\vert _{\mu }\geq \left(
3^{n}-1\right) 5^{-\theta _{\mu }^{\limfunc{doub}}}\left\vert Q\right\vert
_{\mu } \\
&\Longrightarrow &\left\vert Q\right\vert _{\mu }=\left\vert 3Q\right\vert
_{\mu }-\left\vert 3Q\setminus Q\right\vert _{\mu }\leq \left( 1-\frac{%
3^{n}-1}{5^{\theta _{\mu }^{\limfunc{doub}}}}\right) \left\vert
3Q\right\vert _{\mu }\ ,
\end{eqnarray*}%
with a similar inequality for larger $t$. The converse fails since in
particular, reverse doubling measures can vanish on open sets, see Example %
\ref{ex fail} below, while doubling measures cannot.

Let $\sigma $ and $\omega $ be locally finite positive Borel measures on $%
\mathbb{R}^{n}$, and denote by $\mathcal{P}^{n}$ the collection of all cubes
in $\mathbb{R}^{n}$ with sides parallel to the coordinate axes. For $0\leq
\alpha <n$, the classical $\alpha $-fractional Muckenhoupt condition for the
weight pair $\left( \sigma ,\omega \right) $ is given by%
\begin{equation}
A_{2}^{\alpha }\left( \sigma ,\omega \right) \equiv \sup_{Q\in \mathcal{P}%
^{n}}\frac{\left\vert Q\right\vert _{\sigma }}{\left\vert Q\right\vert ^{1-%
\frac{\alpha }{n}}}\frac{\left\vert Q\right\vert _{\omega }}{\left\vert
Q\right\vert ^{1-\frac{\alpha }{n}}}<\infty ,  \label{frac Muck}
\end{equation}%
and the corresponding one-tailed condition by%
\begin{equation}
\mathcal{A}_{2}^{\alpha }\left( \sigma ,\omega \right) \equiv \sup_{Q\in 
\mathcal{Q}^{n}}\mathcal{P}^{\alpha }\left( Q,\sigma \right) \frac{%
\left\vert Q\right\vert _{\omega }}{\left\vert Q\right\vert ^{1-\frac{\alpha 
}{n}}}<\infty ,  \label{one-sided}
\end{equation}%
where the reproducing Poisson integral $\mathcal{P}^{\alpha }$ is given by%
\begin{equation*}
\mathcal{P}^{\alpha }\left( Q,\mu \right) \equiv \int_{\mathbb{R}^{n}}\left( 
\frac{\left\vert Q\right\vert ^{\frac{1}{n}}}{\left( \left\vert Q\right\vert
^{\frac{1}{n}}+\left\vert x-x_{Q}\right\vert \right) ^{2}}\right) ^{n-\alpha
}d\mu \left( x\right) .
\end{equation*}

\subsection{Standard fractional singular integrals, the norm inequality and
testing conditions}

Let $0\leq \alpha <n$ and $\kappa _{1},\kappa _{2}\in \mathbb{N}$. We define
a standard $\left( \kappa _{1}+\delta ,\kappa _{2}+\delta \right) $-smooth $%
\alpha $-fractional CZ kernel $K^{\alpha }(x,y)$ to be a function $K^{\alpha
}:\mathbb{R}^{n}\times \mathbb{R}^{n}\rightarrow \mathbb{R}$ satisfying the
following fractional size and smoothness conditions for some $\delta >0$:
For $x\neq y$, and with $\nabla _{1}$ denoting gradient in the first
variable, and $\nabla _{2}$ denoting gradient in the second variable, 
\begin{eqnarray}
&&\left\vert \nabla _{1}^{j}K^{\alpha }\left( x,y\right) \right\vert \leq
C_{CZ}\left\vert x-y\right\vert ^{\alpha -j-n},\ \ \ \ \ 0\leq j\leq \kappa
_{1},  \label{sizeandsmoothness'} \\
&&\left\vert \nabla _{1}^{\kappa _{1}}K^{\alpha }\left( x,y\right) -\nabla
_{1}^{\kappa _{1}}K^{\alpha }\left( x^{\prime },y\right) \right\vert \leq
C_{CZ}\left( \frac{\left\vert x-x^{\prime }\right\vert }{\left\vert
x-y\right\vert }\right) ^{\delta }\left\vert x-y\right\vert ^{\alpha -\kappa
_{1}-n},\ \ \ \ \ \frac{\left\vert x-x^{\prime }\right\vert }{\left\vert
x-y\right\vert }\leq \frac{1}{2},  \notag
\end{eqnarray}%
and where the same inequalities hold for the adjoint kernel $K^{\alpha ,\ast
}\left( x,y\right) \equiv K^{\alpha }\left( y,x\right) $, in which $x$ and $%
y $ are interchanged, and where $\kappa _{1}$ is replaced by $\kappa _{2}$,
and $\nabla _{1}$\ by $\nabla _{2}$.

If $T^{\alpha }$ is the $\alpha $-fractional singular integral operator
associated with the CZ kernel $K^{\alpha }$, then the norm constant $%
\mathfrak{N}_{T^{\alpha }}=\mathfrak{N}_{T^{\alpha }}\left( \sigma ,\omega
\right) $ is the least constant in the two weight norm inequality%
\begin{equation}
\left( \int_{\mathbb{R}^{n}}\left\vert T^{\alpha }\left( f\sigma \right)
\right\vert ^{2}d\omega \right) ^{\frac{1}{2}}\leq \mathfrak{N}_{T^{\alpha
}}\left( \sigma ,\omega \right) \left( \int_{\mathbb{R}^{n}}\left\vert
f\right\vert ^{2}d\sigma \right) ^{\frac{1}{2}},  \label{norm}
\end{equation}%
taken over all suitable truncations, see e.g. \cite{SaShUr7}.

The \emph{cube testing conditions} associated with an $\alpha $-fractional
singular integral operator $T^{\alpha }$ introduced in \cite{SaShUr7} are
given by%
\begin{eqnarray*}
\left( \mathfrak{T}_{T^{\alpha }}\left( \sigma ,\omega \right) \right) ^{2}
&\equiv &\sup_{Q\in \mathcal{P}^{n}}\frac{1}{\left\vert Q\right\vert
_{\sigma }}\int_{Q}\left\vert T_{\sigma }^{\alpha }\mathbf{1}_{Q}\right\vert
^{2}\omega <\infty , \\
\left( \mathfrak{T}_{\left( T^{\alpha }\right) ^{\ast }}\left( \omega
,\sigma \right) \right) ^{2} &\equiv &\sup_{Q\in \mathcal{P}^{n}}\frac{1}{%
\left\vert Q\right\vert _{\omega }}\int_{Q}\left\vert \left( T_{\sigma
}^{\alpha }\right) ^{\ast }\mathbf{1}_{Q}\right\vert ^{2}\sigma <\infty ,
\end{eqnarray*}

\subsection{The $\mathcal{BICT}$ theorem}

The Bilinear Indicator/Cube Testing property is%
\begin{equation}
\mathcal{BICT}_{T^{\alpha }}\left( \sigma ,\omega \right) \equiv \sup_{Q\in 
\mathcal{P}^{n}}\sup_{E,F\subset Q}\frac{1}{\sqrt{\left\vert Q\right\vert
_{\sigma }\left\vert Q\right\vert _{\omega }}}\left\vert \int_{F}T_{\sigma
}^{\alpha }\left( \mathbf{1}_{E}\right) \omega \right\vert <\infty ,
\label{def ind WBP}
\end{equation}%
where the second supremum is taken over all compact sets $E$ and $F$
contained in a cube $Q$. In \cite{Saw2}, \cite{Saw3} and \cite{Saw4} it is
shown that for doubling weights, the cube testing conditions, the one-tailed
Muckenhoupt conditions, and the bilinear indicator/cube testing property are
sufficient for the norm inequality of an $\alpha $-fractional CZ operator.
In that theorem, the kernel must satisfy smoothness conditions related to
the order of vanishing moments of the weighted Alpert wavelets used (see 
\cite{RaSaWi}), which in turn depend on the doubling exponents of the
weights.

\begin{theorem}[\protect\cite{Saw2}, \protect\cite{Saw3} and \protect\cite%
{Saw4}]
\label{BICT theorem}Suppose that $\sigma $ and $\omega $ are locally finite
positive \emph{doubling} Borel measures on $\mathbb{R}^{n}$. Let $0\leq
\alpha <n$. Suppose also that $T^{\alpha }$ is a standard $\left( \kappa
_{1}+\delta ,\kappa _{2}+\delta \right) $-smooth $\alpha $-fractional Calder%
\'{o}n-Zygmund singular integral in $\mathbb{R}^{n}$, where $\kappa
_{1}>\theta _{\sigma }^{\limfunc{doub}}$ and $\kappa _{2}>\theta _{\omega }^{%
\limfunc{doub}}$ exceed the doubling exponents of $\sigma $ and $\omega $.
In the case $\alpha =0$, we also assume that $T^{0}$ is bounded on
unweighted $L^{2}\left( \mathbb{R}^{n}\right) $. Then%
\begin{equation*}
\mathfrak{N}_{T^{\alpha }}\left( \sigma ,\omega \right) \lesssim \mathfrak{T}%
_{T^{\alpha }}\left( \sigma ,\omega \right) +\mathfrak{T}_{T^{\alpha ,\ast
}}\left( \omega ,\sigma \right) +\mathcal{A}_{2}^{\alpha }\left( \sigma
,\omega \right) +\mathcal{A}_{2}^{\alpha }\left( \omega ,\sigma \right) +%
\mathcal{BICT}_{T^{\alpha }}\left( \sigma ,\omega \right) ,
\end{equation*}%
where the implied constant depends only $\alpha $, $n$, and the doubling
constants of the measures. Moreover, if in addition one of the measures is
an $A_{\infty }$ weight (and if $T^{0}$ is also bounded on \emph{unweighted} 
$L^{2}\left( \mathbb{R}^{n}\right) $ in the case $\alpha =0$), then the
bilinear indicator/cube testing property can be dropped, 
\begin{equation*}
\mathfrak{N}_{T^{\alpha }}\left( \sigma ,\omega \right) \lesssim \mathfrak{T}%
_{T^{\alpha }}\left( \sigma ,\omega \right) +\mathfrak{T}_{T^{\alpha ,\ast
}}\left( \omega ,\sigma \right) +\mathcal{A}_{2}^{\alpha }\left( \sigma
,\omega \right) +\mathcal{A}_{2}^{\alpha }\left( \omega ,\sigma \right) ,
\end{equation*}%
and in terms of cancellation conditions on the kernel $K^{\alpha }\left(
x,y\right) $ of $T^{\alpha }$, we have%
\begin{equation*}
\mathfrak{N}_{T^{\alpha }}\left( \sigma ,\omega \right) \lesssim \mathfrak{A}%
_{K^{\alpha }}\left( \sigma ,\omega \right) +\mathfrak{A}_{K^{\alpha ,\ast
}}\left( \omega ,\sigma \right) +\mathcal{A}_{2}^{\alpha }\left( \sigma
,\omega \right) +\mathcal{A}_{2}^{\alpha }\left( \omega ,\sigma \right) ,
\end{equation*}%
where $\mathfrak{A}_{K^{\alpha }}\left( \sigma ,\omega \right) $ and $%
\mathfrak{A}_{K^{\alpha ,\ast }}\left( \omega ,\sigma \right) $ denote the
least positive constants so that%
\begin{eqnarray}
&&\int_{\left\vert x-x_{0}\right\vert <N}\left\vert \int_{\varepsilon
<\left\vert x-y\right\vert <N}K^{\alpha }\left( x,y\right) d\sigma \left(
y\right) \right\vert ^{2}d\omega \left( x\right) \leq \mathfrak{A}%
_{K^{\alpha }}\left( \sigma ,\omega \right) \ \int_{\left\vert
x_{0}-y\right\vert <N}d\sigma \left( y\right) ,  \label{can cond} \\
&&\ \ \ \ \ \ \ \ \ \ \ \ \ \ \ \ \ \ \ \ \ \ \ \ \ \text{for all }%
0<\varepsilon <N\text{ and }x_{0}\in \mathbb{R}^{n},  \notag
\end{eqnarray}%
along with a similar inequality with constant $\mathfrak{A}_{K^{\alpha ,\ast
}}\left( \omega ,\sigma \right) $, in which the measures $\sigma $ and $%
\omega $ are interchanged and $K^{\alpha }\left( x,y\right) $ is replaced by 
$K^{\alpha ,\ast }\left( x,y\right) =K^{\alpha }\left( y,x\right) $.
\end{theorem}

This theorem raises the following problem.

\begin{problem}
\label{BICT control}Suppose that $\sigma $ and $\omega $ are locally finite
positive doubling Borel measures on $\mathbb{R}^{n}$. Let $0\leq \alpha <n$.
Suppose also that $T^{\alpha }$ is a standard $\alpha $-fractional Calder%
\'{o}n-Zygmund singular integral in $\mathbb{R}^{n}$. Is the two weight
Bilinear Indicator Cube Testing constant $\mathcal{BICT}_{T^{\alpha }}\left(
\sigma ,\omega \right) $ then controlled by the Cube Testing constants $%
\mathfrak{T}_{T^{\alpha }}\left( \sigma ,\omega \right) ,\mathfrak{T}%
_{T^{\alpha ,\ast }}\left( \omega ,\sigma \right) $ and the one-tailed
Muckenhoupt constants $\mathcal{A}_{2}^{\alpha }\left( \sigma ,\omega
\right) ,\mathcal{A}_{2}^{\alpha }\left( \omega ,\sigma \right) $? More
generally, is it true that for every $0<\varepsilon <1$,%
\begin{equation*}
\mathcal{BICT}_{T^{\alpha }}\left( \sigma ,\omega \right) \lesssim \mathfrak{%
T}_{T^{\alpha }}\left( \sigma ,\omega \right) +\mathfrak{T}_{T^{\alpha ,\ast
}}\left( \omega ,\sigma \right) +\mathcal{A}_{2}^{\alpha }\left( \sigma
,\omega \right) +\mathcal{A}_{2}^{\alpha }\left( \omega ,\sigma \right)
+\varepsilon \mathfrak{N}_{T^{\alpha }}\left( \sigma ,\omega \right) ?
\end{equation*}
\end{problem}

\subsection{Main results\label{main}}

In the next section we will give a positive answer to Problem \ref{BICT
control} for $\alpha >0$ and for certain pairs of doubling measures, without
assuming one of them is an $A_{\infty }$ weight. Instead, we assume that the
product measure $\sigma \times \omega $ is \emph{diagonally} reverse
doubling, with a bound on a \emph{diagonal}\footnote{%
`diagonal' since we test over cubes of the form $Q\times Q$ as opposed to $%
Q\times Q^{\prime }$.} reverse doubling exponent $\theta _{\sigma \times
\omega }^{\limfunc{diag}}$, where by definition $\theta _{\sigma \times
\omega }^{\limfunc{diag}}$ satisfies 
\begin{equation*}
\sup_{Q\in \mathcal{P}^{n}}\frac{\left\vert s\left( Q\times Q\right)
\right\vert _{\sigma \times \omega }}{\left\vert Q\times Q\right\vert
_{\sigma \times \omega }}=\sup_{Q\in \mathcal{P}^{n}}\frac{\left\vert
sQ\right\vert _{\sigma }\left\vert sQ\right\vert _{\omega }}{\left\vert
Q\right\vert _{\sigma }\left\vert Q\right\vert _{\omega }}\leq s^{\theta
_{\sigma \times \omega }^{\limfunc{diag}}},\ \ \ \ \ \text{for all
sufficiently small }s>0.
\end{equation*}

\begin{remark}
If $\sigma $ and $\omega $ are reverse doubling with reverse doubling
exponents $\theta _{1}$ and $\theta _{2}$ respectively, then the product
measure $\sigma \times \omega $ is reverse doubling with reverse doubling
exponent $\theta _{1}+\theta _{2}$, hence $\sigma \times \omega $ is
diagonally reverse doubling with exponent $\theta _{\sigma \times \omega }^{%
\limfunc{diag}}\geq \theta _{1}+\theta _{2}$. In particular, if just one of
the measures is reverse doubling, then the product measure is diagonally
reverse doubling with at least half the exponent.
\end{remark}

Actually we prove a bit more, namely that the two weight Bilinear Indicator
Cube Testing constant $\mathcal{BICT}_{I^{\alpha }}\left( \sigma ,\omega
\right) $ for the fractional integral operator $I^{\alpha }$ is controlled
by the classical Muckenhoupt constant $A_{2}^{\alpha }\left( \sigma ,\omega
\right) $ alone in this case. Note that when $\alpha >0$, we have $%
\left\vert T^{\alpha }\nu \right\vert \leq CI^{\alpha }\nu $ for any
positive measure $\nu $, where $I^{\alpha }$ is an example of a smooth $%
\alpha $-fractional Calder\'{o}n-Zygmund singular integral in $\mathbb{R}%
^{n} $. See the next section for more detail.

\begin{theorem}
\label{alpha}Suppose $\sigma $ and $\omega $ are locally finite positive
Borel measures on $\mathbb{R}^{n}$, and that the product measure $\sigma
\times \omega $ is diagonally reverse doubling with a diagonal reverse
doubling exponent $\theta _{\sigma \times \omega }^{\limfunc{diag}}$. Set $%
\theta =\frac{\theta _{\sigma \times \omega }^{\limfunc{diag}}}{2}$. If $%
0<\alpha <n<\theta +\alpha $, then with a constant $C=C_{\theta ,\alpha ,n}$
depending only on $\theta $, $\alpha $, and $n$, we have 
\begin{equation*}
\int_{Q}I^{\alpha }\left( \mathbf{1}_{Q}\sigma \right) d\omega \leq
C_{\theta ,\alpha ,n}\sqrt{A_{2}^{\alpha }\left( \sigma ,\omega \right) }%
\sqrt{\left\vert Q\right\vert _{\sigma }\left\vert Q\right\vert _{\omega }}%
,\ \ \ \ \ \text{for all cubes }Q\in \mathcal{P}^{n}.
\end{equation*}
\end{theorem}

Using Ahlfors-David regular measures, we show that the inequality $n<\theta
+\alpha $ in Theorem \ref{alpha} is sharp. As a corollary of Theorems \ref%
{BICT theorem}\ and \ref{alpha}, we obtain a $T1$ theorem for arbitrary $%
\alpha $-fractional Calder\'{o}n-Zygmund operators in this setting. Note
that Theorem \ref{BICT theorem} requires a degree of smoothness for the
kernel that is related to the \emph{doubling} exponents, as opposed to the 
\emph{reverse doubling} exponents.

\begin{corollary}
Suppose that $\sigma $ and $\omega $ are locally finite positive doubling
Borel measures on $\mathbb{R}^{n}$, with a diagonal reverse doubling
exponent $\theta _{\sigma \times \omega }^{\limfunc{diag}}$ and set $\theta =%
\frac{\theta _{\sigma \times \omega }^{\limfunc{diag}}}{2}$. Suppose $%
0<\alpha <n<\theta +\alpha $ and that $T^{\alpha }$ is a $\left( \kappa
_{1}+\delta ,\kappa _{2}+\delta \right) $-smooth standard $\alpha $%
-fractional Calder\'{o}n-Zygmund singular integral in $\mathbb{R}^{n}$ with $%
\kappa _{1}>\theta _{\sigma }^{\limfunc{doub}}$ and $\kappa _{2}>\theta
_{\omega }^{\limfunc{doub}}$. Then%
\begin{equation*}
\mathfrak{N}_{T^{\alpha }}\left( \sigma ,\omega \right) \lesssim \mathfrak{T}%
_{T^{\alpha }}\left( \sigma ,\omega \right) +\mathfrak{T}_{T^{\alpha ,\ast
}}\omega \left( ,\sigma \right) +\mathcal{A}_{2}^{\alpha }\left( \sigma
,\omega \right) +\mathcal{A}_{2}^{\alpha }\left( \omega ,\sigma \right) ,
\end{equation*}%
where the implied constant depends on $\alpha $, $n$, and the doubling and
reverse doubling exponents for $\sigma $ and $\omega $, and moreover, in
terms of cancellation conditions on the kernel $K^{\alpha }\left( x,y\right) 
$ of $T^{\alpha }$, we have%
\begin{equation*}
\mathfrak{N}_{T^{\alpha }}\left( \sigma ,\omega \right) \lesssim \mathfrak{A}%
_{K^{\alpha }}\left( \sigma ,\omega \right) +\mathfrak{A}_{K^{\alpha ,\ast
}}\left( \omega ,\sigma \right) +\mathcal{A}_{2}^{\alpha }\left( \sigma
,\omega \right) +\mathcal{A}_{2}^{\alpha }\left( \omega ,\sigma \right) ,
\end{equation*}%
where $\mathfrak{A}_{K^{\alpha }}\left( \sigma ,\omega \right) $ and $%
\mathfrak{A}_{K^{\alpha ,\ast }}\left( \omega ,\sigma \right) $ denote the
least positive constants in (\ref{can cond}).\newline
\end{corollary}

In the third section, we will adapt Nazarov's construction from \cite{NaVo}
to give a negative answer to the analogous question for the dyadic Hilbert
transform $H^{\limfunc{dy}}$ (a particular martingale transform) in Theorem %
\ref{alpha}, namely that $H^{\limfunc{dy}}$, which is of course bounded on
unweighted $L^{2}\left( \mathbb{R}\right) $, can fail the inequality%
\begin{equation*}
\left\vert \int_{Q}H^{\limfunc{dy}}\left( \mathbf{1}_{Q}\sigma \right)
d\omega \right\vert \leq C\sqrt{\left\vert Q\right\vert _{\sigma }\left\vert
Q\right\vert _{\omega }},\ \ \ \ \ \text{for all intervals }Q\text{.}
\end{equation*}%
for all positive constants $C$, no matter the doubling constants of $\sigma $
and $\omega $. Let $\mathcal{D}^{0}$ denote the set of dyadic intervals
contained in the unit interval $\left[ 0,1\right] $, and let $H^{\limfunc{dy}%
}$ denote the dyadic Hilbert transform%
\begin{equation}
H^{\limfunc{dy}}\mu \left( x\right) \equiv \frac{1}{2}\sum_{I\in \mathcal{D}%
^{0}:\ x\in I}\bigtriangleup _{I}\mu ,\ \ \ \ \ \bigtriangleup _{I}\mu
\equiv \left( E_{I_{-}}\mu -E_{I_{+}}\mu \right) ,\ \ \ \ \ E_{I}\mu \equiv 
\frac{1}{\left\vert I\right\vert }\int_{I}d\mu ,  \label{def H dyadic}
\end{equation}%
where $I_{-}$ and $I_{+}$ are the left and right hand dyadic children of $I$%
. Note that $H^{\limfunc{dy}}\mu \left( x\right) =\sum_{I\in \mathcal{D}%
^{0}}\left\langle \mu ,h_{I}\right\rangle \frac{1}{\sqrt{\left\vert
I\right\vert }}\mathbf{1}_{I}$ where $\left\{ h_{I}\right\} _{I\in \mathcal{D%
}^{0}}$ is the Haar basis of $L_{0}^{2}\left( \left[ 0,1\right] \right)
\equiv \left\{ f\in L^{2}\left( 0,1\right) :\int_{0}^{1}f=0\right\} $, and
where of course $\mu \left( x\right) =\sum_{I\in \mathcal{D}%
^{0}}\left\langle \mu ,h_{I}\right\rangle h_{I}$ for $\mu \in
L_{0}^{2}\left( 0,1\right) $.

\begin{theorem}[adaptation of \protect\cite{NaVo}]
\label{beta}For every $\Gamma >1$ and $\tau >0$ sufficiently small, there
exist positive weights $u$ and $v$ on the unit interval $\left[ 0,1\right] $
satisfying%
\begin{eqnarray*}
&&\int_{0}^{1}H^{\limfunc{dy}}v\left( x\right) u\left( x\right) dx\geq
\Gamma \sqrt{\left( \int_{0}^{1}u\left( x\right) dx\right) \left(
\int_{0}^{1}v\left( x\right) dx\right) }, \\
&&\left( \frac{1}{\left\vert I\right\vert }\int_{I}u\left( x\right)
dx\right) \left( \frac{1}{\left\vert I\right\vert }\int_{I}v\left( x\right)
dx\right) \leq 1,\ \ \ \ \ \text{for all }I\in \mathcal{D}^{0}, \\
&&1-\tau <\frac{E_{I_{-}}u}{E_{I_{+}}u},\frac{E_{I_{-}}v}{E_{I_{+}}v}<1+\tau
,\ \ \ \ \ \text{for all }I\in \mathcal{D}^{0}.
\end{eqnarray*}
\end{theorem}

From the second line we obtain the two-tailed Muckenhoupt condition $%
\mathcal{A}_{2}\left( u,v\right) \leq C$ for $\tau >0$ sufficiently small,
independent of $\Gamma $, and from the third line, we obtain the doubling
conditions for $u$ and $v$ with doubling constants arbitrarily close to $2$
for $\tau >0$ sufficiently small, independent of $\Gamma $. See \cite{NaVo}
for the routine proofs of these latter assertions.

Finally, in the appendix we discuss one of the main reasons for restricting
our attention to pairs of doubling weights here, and complete the optimal
range for a certain parameter in a characterization of doubling in \cite%
{Saw2}.

\section{Bilinear cube testing for $\protect\alpha >0$}

For $\alpha >0$ we use the domination $T^{\alpha }f\leq CI^{\alpha
}\left\vert f\right\vert $ to obtain%
\begin{equation*}
\left\vert \int_{F}T^{\alpha }\left( \mathbf{1}_{E}\sigma \right) d\omega
\right\vert \leq C\int_{F}I^{\alpha }\left( \mathbf{1}_{E}\sigma \right)
d\omega \leq C\int_{Q}I^{\alpha }\left( \mathbf{1}_{Q}\sigma \right) d\omega
,\ \ \ \ \ \ E,F\subset Q.
\end{equation*}%
Let $\mathfrak{BCT}_{I^{\alpha }}\left( \sigma ,\omega \right) $ denote the
best constant in the Bilinear Cube Testing inequality for the fractional
integral $I^{\alpha }$,%
\begin{equation}
\int_{Q}I^{\alpha }\left( \mathbf{1}_{Q}\sigma \right) d\omega \leq 
\mathfrak{BCT}_{I^{\alpha }}\left( \sigma ,\omega \right) \sqrt{\left\vert
Q\right\vert _{\sigma }\left\vert Q\right\vert _{\omega }},\ \ \ \ \ \text{%
for all cubes }Q\in \mathcal{P}^{n}.  \label{bilinear cube testing}
\end{equation}%
The constant $\mathfrak{BCT}_{I^{\alpha }}\left( \sigma ,\omega \right) $ is
at most the restricted weak type norm constant $\mathfrak{RWT}_{I^{\alpha
}}\left( \sigma ,\omega \right) $ of $I^{\alpha }:L^{2,1}\left( \sigma
\right) \rightarrow L^{2,\infty }\left( \omega \right) $ (which by duality
is the same for the inequality $I^{\alpha }:L^{2,1}\left( \omega \right)
\rightarrow L^{2,\infty }\left( \sigma \right) $), but a characterization of
the restricted weak type constant $\mathfrak{RWT}_{I^{\alpha }}\left( \sigma
,\omega \right) $ has yet to be found. Indeed, the restricted weak type
constant $\mathfrak{RWT}_{I^{\alpha }}\left( \sigma ,\omega \right) $ for $%
I^{\alpha }$ is the smallest constant satisfying%
\begin{equation*}
\int I^{\alpha }\left( f\sigma \right) gd\omega \leq \mathfrak{RWT}%
_{I^{\alpha }}\left( \sigma ,\omega \right) \left\Vert f\right\Vert
_{L^{2,1}\left( \sigma \right) }\left\Vert g\right\Vert _{L^{2,\infty
}\left( \omega \right) },\ \ \ \ \ \text{for all }f\in L^{2,1}\left( \sigma
\right) ,g\in L^{2,\infty }\left( \omega \right) ,
\end{equation*}%
which is in turn equivalent to 
\begin{equation*}
\int_{F}I^{\alpha }\left( \mathbf{1}_{E}\sigma \right) d\omega \leq 
\mathfrak{RWT}_{I^{\alpha }}\left( \sigma ,\omega \right) \sqrt{\left\vert
E\right\vert _{\sigma }\left\vert F\right\vert _{\omega }},\ \ \ \ \ \text{%
for all compact subsets }E,F\subset \mathbb{R}^{n},
\end{equation*}%
by results in Stein and Weiss \cite{StWe2}. Then setting $E=F=Q$ yields (\ref%
{bilinear cube testing}) with $\mathfrak{BCT}_{I^{\alpha }}\left( \sigma
,\omega \right) \leq \mathfrak{RWT}_{I^{\alpha }}\left( \sigma ,\omega
\right) $.

Unfortunately, there is no known \emph{simple}\footnote{%
By \emph{simple} characterization, we mean using conditions of Muckenhoupt
type.} characterization of the harmless looking testing inequality (\ref%
{bilinear cube testing}), and in fact the only known $\emph{simple}$
sufficient condition for (\ref{bilinear cube testing}) to hold is that $%
A_{2}^{\alpha }\left( \sigma ,\omega \right) <\infty $ and one of the
measures is an $A_{\infty }$ weight, see \cite{Saw3}. Since we are assuming $%
A_{2}^{\alpha }\left( \sigma ,\omega \right) <\infty $ in all of our work
above anyways, and since $A_{2}^{\alpha }\left( \sigma ,\omega \right)
<\infty $ is necessary for (\ref{bilinear cube testing}) to hold, we now
consider the problem of characterizing those weight pairs for which $%
\mathfrak{BCT}_{I^{\alpha }}\left( \sigma ,\omega \right) $ is controlled by 
$A_{2}^{\alpha }\left( \sigma ,\omega \right) $, i.e. there is a positive
constant $C$ satisfying%
\begin{equation}
\int_{Q}I^{\alpha }\left( \mathbf{1}_{Q}\sigma \right) d\omega \leq C\sqrt{%
A_{2}^{\alpha }\left( \sigma ,\omega \right) }\sqrt{\left\vert Q\right\vert
_{\sigma }\left\vert Q\right\vert _{\omega }},\ \ \ \ \ \text{for all cubes }%
Q\in \mathcal{P}^{n}.  \label{refinement}
\end{equation}

Again, there does not appear to be a \emph{simple} characterization of (\ref%
{refinement}) either, with the only sufficient condition being that
mentioned above, namely that one of the measures is an $A_{\infty }$ weight.
Theorem \ref{alpha} above provides a different sufficient condition that
involves a diagonal reverse doubling exponent of the product measure $\sigma
\times \omega $.

\subsection{Proof of the diagonal reverse doubling Theorem \protect\ref%
{alpha}}

\begin{proof}
We estimate the left hand side of (\ref{refinement}) by%
\begin{eqnarray*}
\int_{Q}I^{\alpha }\left( \mathbf{1}_{Q}\sigma \right) d\omega
&=&\diint\limits_{Q\times Q}\left\vert x-y\right\vert ^{\alpha -n}d\sigma
\left( x\right) d\omega \left( y\right) \\
&\leq &C_{\alpha ,n}\diint\limits_{Q\times Q}\left\{ \sum_{k=0}^{\infty
}\sum_{I\in \mathcal{D}:\ \ell \left( I\right) =2^{-k}\ell \left( Q\right)
}\ell \left( I\right) ^{\alpha -n}\mathbf{1}_{3I\times 3I}\left( x,y\right)
\right\} d\sigma \left( x\right) d\omega \left( y\right) \\
&=&C_{\alpha ,n}\sum_{k=0}^{\infty }\sum_{\substack{ I\in \mathcal{D}  \\ %
\ell \left( I\right) =2^{-k}\ell \left( Q\right) ,\ I\subset Q}}\left[
2^{-k}\ell \left( Q\right) \right] ^{\alpha -n}\left\vert \left( 3I\times
3I\right) \cap \left( Q\times Q\right) \right\vert _{\sigma \times \omega },
\end{eqnarray*}%
and then using that the diagonal reverse doubling exponent $2\theta $ for $%
\sigma \times \omega $ satisfies $\theta >n-\alpha $, we obtain that for $%
I\subset Q$ with $\ell \left( I\right) =2^{-k}\ell \left( Q\right) $ and $k$
large, 
\begin{equation*}
\sqrt{\left\vert 3I\times 3I\right\vert _{\sigma \times \omega }}=\sqrt{%
\left\vert 2^{-k}\left( 2^{k}3I\times 2^{k}3I\right) \right\vert _{\sigma
\times \omega }}\leq 2^{-k\theta }\sqrt{\left\vert 2^{k}3I\times
2^{k}3I\right\vert _{\sigma \times \omega }}\leq 2^{-k\theta }\sqrt{%
\left\vert 9Q\times 9Q\right\vert _{\sigma \times \omega }}.
\end{equation*}%
Using this estimate for $k$ large, and the crude estimate $\sqrt{\left\vert
3I\times 3I\right\vert _{\sigma \times \omega }}\leq \sqrt{\left\vert
9Q\times 9Q\right\vert _{\sigma \times \omega }}$ for $k$ small, we obtain%
\begin{eqnarray*}
&&\int_{Q}I^{\alpha }\left( \mathbf{1}_{Q}\sigma \right) d\omega \\
&\leq &C_{\alpha ,n}\ell \left( Q\right) ^{\alpha -n}\sqrt{\left\vert
9Q\times 9Q\right\vert _{\sigma \times \omega }}\sum_{k=0}^{\infty
}2^{-k\left( \alpha -n\right) }2^{-k\theta }\sum_{\substack{ I\in \mathcal{D}
\\ \ell \left( I\right) =2^{-k}\ell \left( Q\right) ,\ I\subset Q}}\sqrt{%
\left\vert \left( 3I\times 3I\right) \cap \left( Q\times Q\right)
\right\vert _{\sigma \times \omega }} \\
&\leq &C_{\alpha ,n}\ell \left( Q\right) ^{\alpha -n}\sqrt{\left\vert
9Q\right\vert _{\sigma }\left\vert 9Q\right\vert _{\omega }}%
\sum_{k=0}^{\infty }2^{-k\left( \theta +\alpha -n\right) }\left( \sum 
_{\substack{ I\in \mathcal{D}  \\ \ell \left( I\right) =2^{-k}\ell \left(
Q\right) ,\ I\subset Q}}\left\vert 3I\cap Q\right\vert _{\sigma }\right) ^{%
\frac{1}{2}}\left( \sum_{\substack{ I\in \mathcal{D}  \\ \ell \left(
I\right) =2^{-k}\ell \left( Q\right) ,\ I\subset Q}}\left\vert 3I\cap
Q\right\vert _{\omega }\right) ^{\frac{1}{2}} \\
&\leq &C_{\theta ,\alpha ,n}\ell \left( 9Q\right) ^{\alpha -n}\sqrt{%
\left\vert 9Q\right\vert _{\sigma }\left\vert 9Q\right\vert _{\omega }}\sqrt{%
\left\vert Q\right\vert _{\sigma }\left\vert Q\right\vert _{\omega }}\leq
C_{\theta ,\alpha ,n}\sqrt{A_{2}^{\alpha }}\sqrt{\left\vert Q\right\vert
_{\sigma }\left\vert Q\right\vert _{\omega }}.
\end{eqnarray*}
\end{proof}

\subsection{Sharpness of the diagonal reverse doubling exponent}

Our sharpness examples will be for the equal weight case $\mu =\sigma
=\omega $. We now reformulate the equal weight case of inequality (\ref%
{refinement}) using the semigroup property $I^{\alpha }=I^{\frac{\alpha }{2}%
}\circ I^{\frac{\alpha }{2}}$ and $\beta =\frac{\alpha }{2}$. First, by a
result of Muckenhoupt and Wheeden \cite{MuWh}, we have%
\begin{equation*}
\int_{Q}I^{\alpha }\left( \mathbf{1}_{Q}\mu \right) d\mu =\int_{Q}I^{\frac{%
\alpha }{2}}\circ I^{\frac{\alpha }{2}}\left( \mathbf{1}_{Q}\mu \right) d\mu
=\int_{\mathbb{R}^{n}}I^{\beta }\left( \mathbf{1}_{Q}\mu \right) \left(
x\right) ^{2}dx\approx \int_{\mathbb{R}^{n}}M^{\beta }\left( \mathbf{1}%
_{Q}\mu \right) \left( x\right) ^{2}dx,
\end{equation*}%
where $M^{\beta }\nu \left( x\right) \equiv \sup_{x\in Q}\left\vert
Q\right\vert ^{\frac{\beta }{n}-1}\int_{Q}d\nu $ is the fractional maximal
function. Thus in the equal weight case $\mu =\sigma =\omega $, (\ref%
{refinement}) is equivalent to%
\begin{equation}
\int_{\mathbb{R}^{n}}M^{\beta }\left( \mathbf{1}_{Q}\mu \right) \left(
x\right) ^{2}dx\leq C\sqrt{A_{2}^{\alpha }\left( \mu ,\mu \right) }%
\left\vert Q\right\vert _{\mu },\ \ \ \ \ \text{for all cubes }Q\in \mathcal{%
P}^{n}.  \label{refinement equal}
\end{equation}

\begin{example}
\label{ex fail}In the case $\mu =\sigma =\omega =dx_{1}$ is the singular
measure in the plane $\mathbb{R}^{2}$ given by one-dimensional Lebesgue
measure on the real axis, and with $\alpha =1=\frac{n}{2}$, we have that the
reverse doubling exponent of $\mu \times \mu $ is $2$, and that the
fractional Muckenhoupt constant is finite, yet $\int_{Q}I^{\alpha }\left( 
\mathbf{1}_{Q}\sigma \right) d\omega =\infty $, showing that (\ref%
{refinement equal}) can fail when $\theta =n-\alpha $. Indeed, it is trivial
that $\theta =1=n-\alpha $. For $Q=\left[ 0,R\right] \times \left[ 0,R\right]
$ and $\beta =\frac{\alpha }{2}=\frac{1}{2}$, we have%
\begin{equation*}
M^{\beta }\left( \mathbf{1}_{Q}\mu \right) \left( x_{1},x_{2}\right) \approx
x_{2}^{2\left( \frac{\beta }{2}-1\right) }x_{2}=x_{2}^{\beta -1},\ \ \ \ \
x=\left( x_{1},x_{2}\right) \in Q,
\end{equation*}%
and so%
\begin{equation*}
\frac{1}{\left\vert Q\right\vert _{\mu }}\int_{Q}M^{\beta }\left( \mathbf{1}%
_{Q}\mu \right) \left( x_{1},x_{2}\right) ^{2}dx_{1}dx_{2}\approx \frac{1}{R}%
\int_{0}^{R}\int_{0}^{R}\left( x_{2}^{\beta -1}\right)
^{2}dx_{1}dx_{2}=\int_{0}^{R}x_{2}^{2\beta
-2}dx_{2}=\int_{0}^{R}x_{2}^{-1}dx_{2}=\infty ,
\end{equation*}%
while%
\begin{equation*}
\sqrt{A_{2}^{\alpha }\left( \sigma ,\omega \right) }\approx \sup_{Q=\left[
0,R\right] \times \left[ 0,R\right] }\left\vert Q\right\vert ^{\frac{\alpha 
}{n}-1}\cdot \int_{Q}d\mu =\sup_{R>0}\left( R^{2}\right) ^{\frac{1}{2}%
-1}\cdot R=1.
\end{equation*}
\end{example}

We can extend this sharpness example to general indices $0<\alpha <n$ using
Ahlfors-David regular measures. A measure $\mu $ is said to be Ahlfors-David
regular of order $\theta $ if%
\begin{equation}
\left\vert 3Q\right\vert _{\mu }\approx \ell \left( Q\right) ^{\theta }\text{
whenever }\left\vert Q\right\vert _{\mu }>0.  \label{def AD}
\end{equation}

\begin{lemma}
\label{AD lemma}If $\mu $ is any Ahlfors-David regular measure in $\mathbb{R}%
^{n}$ of order $n-\alpha $ where $0<\alpha <n$, then (\ref{refinement equal}%
) fails with $\beta =\frac{\alpha }{2}$.
\end{lemma}

\begin{proof}
Suppose that $\mu $ is Ahlfors-David regular of order $\theta $. First we
note that%
\begin{equation*}
\sqrt{A_{2}^{\alpha }\left( \mu ,\mu \right) }\approx \sup_{Q\in \mathcal{P}%
^{n}}\left\vert Q\right\vert ^{\frac{\alpha }{n}-1}\int_{Q}d\mu \approx
\sup_{Q\in \mathcal{P}^{n}}\ell \left( Q\right) ^{\alpha -n}\ell \left(
Q\right) ^{\theta }=1,
\end{equation*}%
if $\theta =n-\alpha $. To show that the left side of (\ref{refinement equal}%
) is infinite for the same choice of $\theta $, we proceed in four steps.
Let $\mathfrak{C}^{\left( N\right) }\left( Q\right) $ denote the collection
of dyadic subcubes $Q^{\prime }$ of $Q$ having side length $\ell \left(
Q^{\prime }\right) =2^{-N}\ell \left( Q\right) $. Throughout the proof,
constants implied by $\approx $ and $\lesssim $ depend only on $\alpha $, $n$
and the Ahlfors-David constants implicit in the definition (\ref{def AD}).

\textbf{Step 1}: Let 
\begin{equation*}
\Gamma _{N}\left( Q\right) \equiv \left\{ Q^{\prime }\in \mathfrak{C}%
^{\left( N\right) }\left( Q\right) :\left\vert Q^{\prime }\right\vert _{\mu
}>0\right\} ,\ \ \ \ \ \text{for }Q\in \mathcal{P}^{n}.
\end{equation*}%
Since $\mu $ is Ahlfors-David regular of order $\theta \equiv n-\alpha $, we
have for any cube $Q\in \mathcal{P}^{n}$ that both%
\begin{eqnarray*}
\sum_{Q^{\prime }\in \Gamma _{N}\left( Q\right) }\left\vert 3Q^{\prime
}\right\vert _{\mu } &\approx &\sum_{Q^{\prime }\in \Gamma _{N}\left(
Q\right) }\ell \left( Q^{\prime }\right) ^{\theta }=\#\Gamma _{N}\left(
Q\right) \cdot 2^{-N\theta }\ell \left( Q\right) ^{\theta }, \\
\sum_{Q^{\prime }\in \Gamma _{N}\left( Q\right) }\left\vert 3Q^{\prime
}\right\vert _{\mu } &\lesssim &\left\vert 3Q\right\vert _{\mu }\ .
\end{eqnarray*}%
Thus we obtain%
\begin{equation*}
\#\Gamma _{N}\left( Q\right) \cdot 2^{-N\theta }\ell \left( Q\right)
^{\theta }\lesssim \left\vert 3Q\right\vert _{\mu }\approx \ell \left(
Q\right) ^{\theta },\ \ \ \ \ \text{if }\left\vert Q\right\vert _{\mu }>0,
\end{equation*}%
and hence%
\begin{equation}
\#\Gamma _{N}\left( Q\right) \lesssim 2^{N\theta },\ \ \ \ \ \text{if }Q\in 
\mathcal{P}^{n}.  \label{card Gamma}
\end{equation}%
In particular there is $N=N_{n,\alpha ,\mu }$ sufficiently large that $%
\mathfrak{C}^{\left( N\right) }\left( Q\right) \setminus \Gamma _{N}\left(
Q\right) \neq \emptyset $ for all cubes $Q\in \mathcal{P}^{n}$.

\textbf{Step 2}: Fix a cube $Q$ and let $N=N_{n,\alpha ,\mu }$ be as in Step
1. Then $\mathfrak{C}^{\left( N\right) }\left( Q\right) \setminus \Gamma
_{N}\left( Q\right) \neq \emptyset $ and so there is $Q^{\ast }\in \mathfrak{%
C}^{\left( N\right) }\left( Q\right) $ with $\left\vert Q^{\ast }\right\vert
_{\mu }=0$. Since 
\begin{equation*}
\inf_{x\in Q}M^{\beta }\left( \mathbf{1}_{Q}\mu \right) \left( x\right) \geq
\ell \left( Q\right) ^{\beta -n}\int_{Q}d\mu \ ,
\end{equation*}%
we then have%
\begin{equation*}
\int_{Q^{\ast }}M^{\beta }\left( \mathbf{1}_{Q}\mu \right) \left( x\right)
^{2}dx\geq \ell \left( Q\right) ^{\alpha -2n}\left( \int_{Q}d\mu \right)
^{2}\ell \left( Q^{\ast }\right) ^{n}\approx 2^{-Nn}\ell \left( Q\right)
^{\theta +\alpha -n}\int_{Q}d\mu .
\end{equation*}%
Set $\Omega _{1}\left( Q\right) \equiv Q^{\ast }$. Since $\theta +\alpha
-n=0 $, there is a positive constant $c_{N}$ such that for $Q\in \mathcal{P}%
^{n}$,%
\begin{equation}
\int_{\Omega _{1}\left( Q\right) }M^{\beta }\left( \mathbf{1}_{Q}\mu \right)
\left( x\right) ^{2}dx\geq c_{N}\int_{Q}d\mu \ .  \label{lower bound}
\end{equation}

\textbf{Step 3}: Again fix a cube $Q$ and let $N=N_{n,\alpha ,\mu }$ be as
in Step 1. Let $\Gamma _{N}\left( Q\right) =\left\{ Q_{k}\right\} _{k=1}^{K}$
where $K\lesssim 2^{N\theta }$ by (\ref{card Gamma}). Then we apply Step 2
to the cube $Q_{k}$ to obtain a cube $Q_{k}^{\ast }$ with $\left\vert
Q_{k}^{\ast }\right\vert _{\mu }=0$ and%
\begin{equation*}
\int_{Q_{k}^{\ast }}M^{\beta }\left( \mathbf{1}_{Q_{k}}\mu \right) \left(
x\right) ^{2}dx\geq c_{N}\int_{Q_{k}}d\mu .
\end{equation*}%
Then with $\Omega _{2}\equiv \dbigcup\limits_{k=1}^{K}Q_{k}^{\ast }$, we
obtain upon summing in $k$ that%
\begin{equation*}
\int_{\Omega _{2}}M^{\beta }\left( \mathbf{1}_{Q}\mu \right) \left( x\right)
^{2}dx\geq c_{N}\int_{Q}d\mu .
\end{equation*}%
Note that $Q_{k}^{\ast }\subset Q_{k}$ where $\left\vert Q_{k}\right\vert
_{\mu }>0$, and that $\left\vert Q^{\ast }\right\vert _{\mu }=0$, which
shows that $Q_{k}^{\ast }\cap Q^{\ast }=\emptyset $ for all $k$, hence $%
\Omega _{1}\cap \Omega _{2}=\emptyset $. Thus we have that 
\begin{equation*}
\int_{\Omega _{1}\cup \Omega _{2}}M^{\beta }\left( \mathbf{1}_{Q}\mu \right)
\left( x\right) ^{2}dx\geq 2c_{N}\int_{Q}d\mu .
\end{equation*}

\textbf{Step 4}: Now repeat Step 3 indefinitely to obtain%
\begin{equation*}
\int_{\Omega _{1}\cup \Omega _{2}\cup ...\cup \Omega _{m}}M^{\beta }\left( 
\mathbf{1}_{Q}\mu \right) \left( x\right) ^{2}dx\geq mc_{N}\int_{Q}d\mu ,\ \
\ \ \ \text{for all }m\geq 1,
\end{equation*}%
which of course shows that%
\begin{equation*}
\int_{Q}M^{\beta }\left( \mathbf{1}_{Q}\mu \right) \left( x\right)
^{2}dx=\infty .
\end{equation*}
\end{proof}

\begin{problem}
The measures $\mu $ in the sharpness examples above are not however
doubling, only reverse doubling. This begs the question of whether or not (%
\ref{refinement}) can hold for all pairs of \emph{doubling} measures, a
question we leave open.
\end{problem}

Finally, Lemma \ref{AD lemma} shows the failure of the trace inequality $I^{%
\frac{\alpha }{2}}:L^{2}\rightarrow L^{2}\left( \partial \Omega \right) $
for a domain $\Omega \subset \mathbb{R}^{n}$ when $\partial \Omega $ is an
Ahlfors-David regular set of order $n-\alpha $. For example $I^{\frac{1}{2}%
}:L^{2}\rightarrow L^{2}\left( \partial \Omega \right) $ fails in the plane
if $\partial \Omega $ is the Cantor dust fractal - Example \ref{ex fail} is
the case when $\partial \Omega $ is a line.

\section{Failure of $\mathcal{BCT}$ for the dyadic Hilbert transform}

We do not know if the analogous inequality for the Hilbert transform on the
real line, i.e.%
\begin{equation*}
\int_{Q}\left\vert H\left( \mathbf{1}_{Q}\sigma \right) \right\vert d\omega
\leq C\sqrt{A_{2}\left( \sigma ,\omega \right) }\sqrt{\left\vert
Q\right\vert _{\sigma }\left\vert Q\right\vert _{\omega }},\ \ \ \ \ \text{%
for all intervals }Q,
\end{equation*}%
holds, but we can show that the analogous question for the dyadic Hilbert
transform is answered in the negative here (no it can fail) using an
adaptation of Nazarov's Bellman construction in \cite{NaVo}.

The following Bilinear Cube Testing condition for the Hilbert transform $H$
is of course implied by restricted weak type for $H$:%
\begin{equation}
\left\vert \int_{Q}H\left( \mathbf{1}_{Q}\sigma \right) d\omega \right\vert
\leq \mathcal{BCT}_{H}\sqrt{\left\vert Q\right\vert _{\sigma }\left\vert
Q\right\vert _{\omega }},\ \ \ \ \ \text{for all intervals }Q\text{.}
\label{BCT}
\end{equation}%
Unfortunately we are unable to determine if $\mathcal{BCT}_{H}<\infty $.
Instead, we will prove here Theorem \ref{beta}, that shows the discrete
dyadic form of the inequality fails, i.e. that the inequality%
\begin{equation*}
\left\vert \int_{Q}H^{\limfunc{dy}}\left( \mathbf{1}_{Q}\sigma \right)
d\omega \right\vert \leq \mathcal{BCT}_{H^{\limfunc{dy}}}\sqrt{\left\vert
Q\right\vert _{\sigma }\left\vert Q\right\vert _{\omega }},\ \ \ \ \ \text{%
for all dyadic intervals }Q\subset \left[ 0,1\right) \text{,}
\end{equation*}%
fails. In fact, Theorem \ref{beta} is an easy consequence of (\ref{def H
dyadic}) and the following simpler variant of a Bellman construction from 
\cite{NaVo}.

\subsection{The dyadic Bellman construction}

\begin{lemma}
\label{dyadic restricted}Let $0<\tau <1$. Then for every $\Gamma >1$, there
exists a pair of weights $\left( U,V\right) $ on the unit interval $%
I^{0}\equiv \left[ 0,1\right] $, and a positive integer $M\in \mathbb{N}$,
such that each of the functions $U,V$ is positive on $\left[ 0,1\right] $
and constant\footnote{%
We do not actually need this constant property here since we are unable to
apply the `supervisor' argument from \cite{NaVo} to extend the
counterexample to the $\alpha $-fractional Riesz transform on the line when $%
\alpha >0.$} on every interval $K\in \mathcal{D}^{0}$ having side length $%
2^{-M}$, and moreover,%
\begin{eqnarray*}
&&\sum_{I\in \mathcal{D}^{0}}\left( \bigtriangleup _{I}V\right) \left(
E_{I}U\right) \left\vert I\right\vert >\Gamma \sqrt{\left( E_{I^{0}}U\right)
\left( E_{I^{0}}V\right) }, \\
&&\left( E_{I}U\right) \left( E_{I}V\right) \leq 1,\ \ \ \ \ \text{for all }%
I\in \mathcal{D}^{0}, \\
&&1-\tau <\frac{E_{I_{-}}U}{E_{I_{+}}U},\frac{E_{I_{-}}V}{E_{I_{+}}V}<1+\tau
,\ \ \ \ \ \text{for all }I\in \mathcal{D}^{0}.
\end{eqnarray*}
\end{lemma}

To prove this lemma we use the Bellman function%
\begin{equation}
\mathcal{B}\left( x\right) \equiv \sup_{J\in \mathcal{D}^{0}}\left\{ \frac{1%
}{\left\vert J\right\vert }\sum_{I\in \mathcal{D}^{0}:\ I\subset J}\left(
\bigtriangleup _{I}V\right) \left( E_{I}U\right) \left\vert I\right\vert
:\left( U,V\right) \in \mathcal{F}_{J;x}\right\} ,  \label{linear Bell}
\end{equation}%
for $x=\left( x_{1},x_{2}\right) \in \left( 0,\infty \right) ^{2}$ with $%
x_{1}x_{2}<1$, in analogy with that in \cite{NaVo}, where $\mathcal{F}_{J;x}$
consists of those pairs $\left( U,V\right) $ of positive functions on $J$
such that%
\begin{eqnarray*}
E_{J}U &=&x_{1},\ \ \ E_{J}V=x_{2}, \\
\text{and }\left( E_{I}U\right) \left( E_{I}V\right) &<&1,\ \ \ \ \ \text{%
for all }I\in \mathcal{D}^{0}\text{ with }I\subset J.
\end{eqnarray*}%
Note that the averages of $U$ and $V$ are only fixed to be $x_{1}$ and $%
x_{2} $ respectively at the interval $J$. Moreover, while it is the case
that $\bigtriangleup _{I}V$ can be negative, an appropriate switching of
children for each parent replaces $\bigtriangleup _{I}V$ with $\left\vert
\bigtriangleup _{I}V\right\vert $ while leaving $E_{I}U$ unaffected, and so
we also have%
\begin{equation*}
\mathcal{B}\left( x\right) \equiv \sup_{J\in \mathcal{D}^{0}}\left\{ \frac{1%
}{\left\vert J\right\vert }\sum_{I\in \mathcal{D}^{0}:\ I\subset
J}\left\vert \bigtriangleup _{I}V\right\vert \left( E_{I}U\right) \left\vert
I\right\vert :\left( U,V\right) \in \mathcal{F}_{J;x}\right\} ,
\end{equation*}%
which shows in particular that $\mathcal{B}\left( x\right) $ is positive.

The Bellman function $\mathcal{B}\left( x\right) $ satisfies the rescaling
property,%
\begin{equation}
\frac{1}{\left\vert \widehat{J}\right\vert }\sum_{I\in \mathcal{D}^{0}:\
I\subset \widehat{J}}\left\vert \bigtriangleup _{I}\widehat{V}\right\vert
\left( E_{I}\widehat{U}\right) \left\vert I\right\vert =\frac{1}{\left\vert
J\right\vert }\sum_{I\in \mathcal{D}^{0}:\ I\subset J}\left\vert
\bigtriangleup _{I}V\right\vert \left( E_{I}U\right) \left\vert I\right\vert
,  \label{rescale}
\end{equation}%
where $\left( \widehat{U},\widehat{V}\right) =\left(
S_{a,b}U,S_{a,b}V\right) \in \mathcal{F}_{J;x}$ with $S_{a,b}f\left(
z\right) =f\left( T_{a,b}^{-1}z\right) $ and $T_{a,b}y=ay+b$, and where $%
\widehat{J}=T_{a,b}J$ with $a>0$ and $b\in \mathbb{R}$. Indeed, the affine
map $T_{a,b}$ takes an interval $I$ to an interval $T_{a,b}I$ with $%
\left\vert T_{a,b}I\right\vert =a\left\vert I\right\vert $, and preserves
the dyadic structures within the intervals $I$ and $T_{a,b}I$. Moreover, if $%
a=2^{k}$ and $b=2^{k}\ell $ for some $k\in \mathbb{Z}$ and $\ell \in \mathbb{%
N}$, then $I\in \mathcal{D}$ if and only if $T_{a,b}I\in \mathcal{D}$. Note
that $S_{a,b}$ takes functions $f$ supported in $I$ to functions $S_{a,b}f$
supported in $T_{a,b}I$, and moreover preserves averages over all dyadic
intervals $I$, i.e. 
\begin{equation*}
E_{T_{a,b}I}\left( S_{a,b}f\right) =\frac{1}{\left\vert T_{a,b}I\right\vert }%
\int_{T_{a,b}I}f\left( T_{a,b}^{-1}z\right) dz=\frac{1}{\left\vert
T_{a,b}I\right\vert }\int_{I}f\left( y\right) ady=\frac{1}{a\left\vert
I\right\vert }\int_{I}f\left( y\right) ady=E_{I}f,
\end{equation*}%
as well as the `difference averages',%
\begin{eqnarray*}
\bigtriangleup _{T_{a,b}I}\left( S_{a,b}f\right) &=&E_{\left(
T_{a,b}I\right) _{-}}\left( S_{a,b}f\right) -E_{\left( T_{a,b}I\right)
_{+}}\left( S_{a,b}f\right) \\
&=&\frac{1}{\left\vert \left( T_{a,b}I\right) _{-}\right\vert }\int_{\left(
T_{a,b}I\right) _{-}}S_{a,b}f\left( z\right) dz-\frac{1}{\left\vert \left(
T_{a,b}I\right) _{+}\right\vert }\int_{\left( T_{a,b}I\right)
_{+}}S_{a,b}f\left( z\right) dz \\
&=&\frac{1}{\left\vert \left( T_{a,b}I\right) _{-}\right\vert }\int_{\left(
T_{a,b}I\right) _{-}}f\left( T_{a,b}^{-1}z\right) dz-\frac{1}{\left\vert
\left( T_{a,b}I\right) _{+}\right\vert }\int_{\left( T_{a,b}I\right)
_{+}}f\left( T_{a,b}^{-1}z\right) dz \\
&=&\frac{1}{a\left\vert I_{-}\right\vert }\int_{I_{-}}f\left( y\right) ady-%
\frac{1}{a\left\vert I_{+}\right\vert }\int_{I_{+}}f\left( y\right)
ady=E_{I_{-}}f-E_{I_{+}}f=\bigtriangleup _{I}f.
\end{eqnarray*}

Now fix dyadic intervals $J$ and $\widehat{J}$ in $\mathcal{D}^{0}$. Choose
an affine map $T_{a,b}$ with $a=2^{k}$ and $b=2^{k}\ell $, for some $k,\ell
\in \mathbb{Z}$, that takes the interval $J$ one-to-one and onto the
interval $\widehat{J}=T_{a,b}J$. Define functions $\widehat{U}=S_{a,b}U$ and 
$\widehat{V}=S_{a,b}V$. Then we have%
\begin{eqnarray*}
&&\frac{1}{\left\vert \widehat{J}\right\vert }\sum_{I\in \mathcal{D}^{0}:\
I\subset \widehat{J}}\left\vert \bigtriangleup _{I}\widehat{V}\right\vert
\left( E_{I}\widehat{U}\right) \left\vert I\right\vert =\frac{1}{\left\vert
T_{a,b}J\right\vert }\sum_{I\in \mathcal{D}^{0}:\ I\subset
T_{a,b}J}\left\vert \bigtriangleup _{I}\left( S_{a,b}V\right) \right\vert
E_{I}\left( S_{a,b}U\right) \left\vert I\right\vert \\
&=&\frac{1}{\left\vert T_{a,b}J\right\vert }\sum_{I\in \mathcal{D}^{0}:\
I\subset J}\left\vert \bigtriangleup _{T_{a,b}I}\left( S_{a,b}V\right)
\right\vert \left( E_{T_{a,b}I}\left( S_{a,b}U\right) \right) \left\vert
T_{a,b}I\right\vert \\
&=&\frac{1}{a\left\vert J\right\vert }\sum_{I\in \mathcal{D}^{0}:\ I\subset
J}\left\vert \bigtriangleup _{I}V\right\vert \left( E_{I}U\right)
a\left\vert I\right\vert =\frac{1}{\left\vert J\right\vert }\sum_{I\in 
\mathcal{D}^{0}:\ I\subset J}\left\vert \bigtriangleup _{I}V\right\vert
\left( E_{I}U\right) \left\vert I\right\vert ,
\end{eqnarray*}%
and also $\left( \widehat{U},\widehat{V}\right) =\left(
S_{a,b}U,S_{a,b}V\right) \in \mathcal{F}_{\widehat{J};x}$ since%
\begin{equation*}
E_{\widehat{J}}\left( \widehat{U}\right) =E_{T_{a,b}J}\left( S_{a,b}U\right)
=E_{J}U=x_{1}\text{ and }E_{\widehat{J}}\left( \widehat{V}\right)
=E_{T_{a,b}J}\left( S_{a,b}V\right) =E_{J}V=x_{2}\ .
\end{equation*}

Now let%
\begin{equation}
\Omega \equiv \left\{ x=\left( x_{1},x_{2}\right) \in \left( 0,\infty
\right) ^{2}:x_{1}x_{2}<1\right\} .  \label{def Omega}
\end{equation}%
Assuming that $\mathcal{B}\left( x\right) <\infty $ for all $x\in \Omega $,
we will derive a contradiction from Theorem \ref{quant bil} below, thus
concluding that $\mathcal{B}\left( x\right) $ must be $\infty $ for some $%
x\in \Omega $, and so in particular that sup$_{x\in \Omega }\frac{\mathcal{B}%
\left( x\right) }{\sqrt{x_{1}x_{2}}}=\infty $. In any event, this shows that
for any $\Gamma >1$ there is $x\in \Omega $, $J\in \mathcal{D}^{0}$ and $%
\left( U,V\right) \in \mathcal{F}_{J;x}$ such that%
\begin{equation*}
\frac{1}{\left\vert J\right\vert \sqrt{\left( E_{J}U\right) \left(
E_{J}U\right) }}\sum_{I\in \mathcal{D}^{0}:\ I\subset J}\left(
\bigtriangleup _{I}V\right) \left( E_{I}U\right) \left\vert I\right\vert
>\Gamma ,
\end{equation*}%
which if $J=I^{0}$, as we may assume, gives%
\begin{equation*}
\sum_{I\in \mathcal{D}^{0}}\left\vert \bigtriangleup _{I}V\right\vert \left(
E_{I}U\right) \left\vert I\right\vert >\Gamma \sqrt{\left( E_{I^{0}}U\right)
\left( E_{I^{0}}U\right) },
\end{equation*}%
since $\left\vert I^{0}\right\vert =1$. This will complete the proof of
Lemma \ref{dyadic restricted} upon restricting the sum of the nonnegative
terms $\left\vert \bigtriangleup _{I}V\right\vert \left( E_{I}U\right)
\left\vert I\right\vert $ for $I\in \mathcal{D}^{0}$ to intervals $I$ of
side length at least $2^{-M}$ for a sufficiently large $M\in \mathbb{N}$.

We begin by establishing a very strict concavity property of $\mathcal{B}%
\left( x\right) $ in $\Omega $.

\begin{theorem}
\label{quant bil}Assume that $\mathcal{B}\left( x\right) <\infty $ for all\ $%
x\in \Omega $. If $y=\left( y_{1},y_{2}\right) $ is such that $x,x+y,x-y\in
\Omega $, then%
\begin{equation*}
\frac{\mathcal{B}\left( x+y\right) +\mathcal{B}\left( x-y\right) }{2}%
+2\left\vert y_{2}\right\vert x_{1}-\mathcal{B}\left( x\right) \leq 0.
\end{equation*}
\end{theorem}

\begin{proof}
Fix an interval $J\in \mathcal{D}^{0}$, which we could of course take to be $%
I^{0}=\left[ 0,1\right) $. Consider two pairs $\left( U_{+},V_{+}\right) $
and $\left( U_{-},V_{-}\right) $ with corresponding intervals $J_{x+y}$ and $%
J_{x-y}$ that are `$\eta $-maximizing' for $x+y$ and $x-y$ respectively with 
$\eta >0$, by which we mean that%
\begin{eqnarray*}
\mathcal{B}\left( x+y\right) -\eta &<&\frac{1}{\left\vert J_{x+y}\right\vert 
}\sum_{I\in \mathcal{D}^{0}:\ I\subset J_{x+y}}\left\vert \bigtriangleup
_{I}V_{+}\right\vert \left( E_{I}U_{+}\right) \left\vert I\right\vert ,\ \ \
\ \ \text{for }E_{J_{x+y}}U_{+}=x_{1}+y_{1},E_{J_{x+y}}V_{+}=x_{2}+y_{2}, \\
\mathcal{B}\left( x-y\right) -\eta &<&\frac{1}{\left\vert J_{x-y}\right\vert 
}\sum_{I\in \mathcal{D}^{0}:\ I\subset J_{x-y}}\left\vert \bigtriangleup
_{I}V_{-}\right\vert \left( E_{I}U_{-}\right) \left\vert I\right\vert ,\ \ \
\ \ \text{for }E_{J_{x-y}}U_{-}=x_{1}-y_{1},E_{J_{x-y}}V_{-}=x_{2}-y_{2}.
\end{eqnarray*}%
Moreover, we may assume that all of the weights above are constant on
sufficiently small intervals. By rescaling with appropriate maps $T_{a,b}$
and $S_{a,b}$ as in (\ref{rescale}) above, we may suppose that the dyadic
intervals $J_{x+y},J_{x-y}$ have the form $J_{+},J_{-}$ respectively, where $%
J$ is the interval fixed at the beginning of the proof, and moreover that $%
U_{\pm },V_{\pm }$ are supported in $J_{\pm }$.

Following \cite{NaVo}\ we now construct a pair $\left( \widetilde{U},%
\widetilde{V}\right) $ supported in $J$ satisfying 
\begin{equation*}
\widetilde{U}\equiv \left\{ 
\begin{array}{ccc}
U_{+} & \text{ on } & J_{+} \\ 
U_{-} & \text{ on } & J_{-} \\ 
0 & \text{ on } & J^{c}%
\end{array}%
\right. \text{ and }\widetilde{V}\equiv \left\{ 
\begin{array}{ccc}
V_{+} & \text{ on } & J_{+} \\ 
V_{-} & \text{ on } & J_{-} \\ 
0 & \text{ on } & J^{c}%
\end{array}%
\right. .
\end{equation*}%
We claim that $\left( \widetilde{U},\widetilde{V}\right) \in \mathcal{F}%
_{J;x}$. Indeed,%
\begin{eqnarray*}
E_{J}\widetilde{U} &=&\frac{1}{\left\vert J\right\vert }\int_{J}\widetilde{U}%
\left( x\right) dx=\frac{1}{\left\vert J\right\vert }\int_{J_{+}}U_{+}\left(
x\right) dx+\frac{1}{\left\vert J\right\vert }\int_{J_{-}}U_{-}\left(
x\right) dx \\
&=&\frac{1}{2}\left\{ \frac{1}{\left\vert J_{+}\right\vert }%
\int_{J_{+}}U_{+}\left( x\right) dx+\frac{1}{\left\vert J_{-}\right\vert }%
\int_{J_{-}}U_{-}\left( x\right) dx\right\} \\
&=&\frac{1}{2}\left\{ E_{J_{+}}\widetilde{U}+E_{J_{-}}\widetilde{U}\right\} =%
\frac{1}{2}\left\{ x_{1}+y_{1}+x_{1}-y_{1}\right\} =x_{1}\ ,
\end{eqnarray*}%
and similarly $E_{J}\widetilde{V}=x_{2}$, and of course then%
\begin{equation*}
\left( E_{J}\widetilde{U}\right) \left( E_{J}\widetilde{V}\right)
=x_{1}x_{2}<1.
\end{equation*}%
Turning next to the strict dyadic subintervals $I$ of $J$ we have for $%
I\subset J_{+}$,%
\begin{eqnarray*}
E_{I}\widetilde{U} &=&E_{I}U_{+}\ ,\ \ \ \bigtriangleup _{I}\widetilde{U}%
=\bigtriangleup _{I}U_{+}\ , \\
E_{I}\widetilde{V} &=&E_{I}V_{+}\ ,\ \ \ \bigtriangleup _{I}\widetilde{V}%
=\bigtriangleup _{I}V_{+}\ ,
\end{eqnarray*}%
and for $I\subset J_{-}$,%
\begin{eqnarray*}
E_{I}\widetilde{U} &=&E_{I}U_{-}\ ,\ \ \ \bigtriangleup _{I}\widetilde{U}%
=\bigtriangleup _{I}U_{-}\ , \\
E_{I}\widetilde{V} &=&E_{I}V_{-}\ ,\ \ \ \bigtriangleup _{I}\widetilde{V}%
=\bigtriangleup _{I}V_{-}\ .
\end{eqnarray*}%
Consequently we obtain%
\begin{equation*}
\left( E_{I}\widetilde{U}\right) \left( E_{I}\widetilde{V}\right) <1,
\end{equation*}%
which completes the proof of our claim that $\left( \widetilde{U},\widetilde{%
V}\right) \in \mathcal{F}_{J;x}$.

Note that we also have%
\begin{equation*}
\bigtriangleup _{J}\widetilde{V}=E_{J_{-}}\widetilde{V}-E_{J_{+}}\widetilde{V%
}=E_{J_{-}}V_{-}-E_{J_{+}}V_{+}=\left[ \left( x_{2}-y_{2}\right) -\left(
x_{2}+y_{2}\right) \right] =-2y_{2}\ .
\end{equation*}%
Then with 
\begin{equation*}
\mathcal{L}_{J}\left( f,g\right) \equiv \sum_{I\in \mathcal{D}^{0}:\
I\subset J}\left\vert \bigtriangleup _{I}g\right\vert \left( E_{I}f\right)
\left\vert I\right\vert ,\ \ \ \ \ \text{for }\left( f,g\right) \in \mathcal{%
F}_{J;x}\ ,
\end{equation*}%
we have%
\begin{eqnarray*}
&&\mathcal{B}\left( x\right) \geq \frac{1}{\left\vert J\right\vert }\mathcal{%
L}_{J}\left( \widetilde{U},\widetilde{V}\right) =\left\vert \bigtriangleup
_{J}\widetilde{V}\right\vert \left( E_{J}\widetilde{U}\right) +\frac{1}{%
\left\vert J\right\vert }\sum_{I\in \mathcal{D}^{0}:\ I\subset
J_{+}}\left\vert \bigtriangleup _{I}\widetilde{V}\right\vert \left( E_{I}%
\widetilde{U}\right) \left\vert I\right\vert +\frac{1}{\left\vert
J\right\vert }\sum_{I\in \mathcal{D}^{0}:\ I\subset J_{-}}\left\vert
\bigtriangleup _{I}\widetilde{V}\right\vert \left( E_{I}\widetilde{U}\right)
\left\vert I\right\vert \\
&=&2\left\vert y_{2}\right\vert x_{1}+\frac{1}{2}\frac{1}{\left\vert
J_{+}\right\vert }\sum_{I\in \mathcal{D}^{0}:\ I\subset J_{+}}\left\vert
\bigtriangleup _{I}V_{+}\right\vert \left( E_{I}U_{+}\right) \left\vert
I\right\vert +\frac{1}{2}\frac{1}{\left\vert J_{-}\right\vert }\sum_{I\in 
\mathcal{D}^{0}:\ I\subset J_{-}}\left\vert \bigtriangleup
_{I}V_{-}\right\vert \left( E_{I}U_{-}\right) \left\vert I\right\vert \\
&>&2\left\vert y_{2}\right\vert x_{1}+\frac{1}{2}\left\{ \mathcal{B}\left(
x+y\right) -\eta +\mathcal{B}\left( x-y\right) -\eta \right\} =2\left\vert
y_{2}\right\vert x_{1}+\frac{\mathcal{B}\left( x+y\right) +\mathcal{B}\left(
x-y\right) }{2}-\eta .
\end{eqnarray*}%
Since $\eta >0$ is arbitrary, this gives%
\begin{equation*}
\frac{\mathcal{B}\left( x+y\right) +\mathcal{B}\left( x-y\right) }{2}%
+2\left\vert y_{2}\right\vert x_{1}-\mathcal{B}\left( x\right) \leq 0,
\end{equation*}%
and this completes the proof of Theorem \ref{quant bil}.
\end{proof}

We may assume that $\mathcal{B}\left( x\right) $ is finite everywhere on $%
\Omega $, since otherwise we are done. Then Theorem \ref{quant bil} shows in
particular that $\mathcal{B}\left( x\right) $ is concave on $\Omega $, and
so by a result of Buseman and Feller \cite{BuFe} (extended to $\mathbb{R}%
^{n} $ by Alexandrov \cite{Ale}), $\mathcal{B}\left( x\right) $ is
differentiable to second order for almost every $x\in \Omega $. But if the
Bellman function $\mathcal{B}$ is twice differentiable at a fixed $x\in
\Omega $, Taylor's formula gives%
\begin{eqnarray*}
\mathcal{B}\left( x\pm y\right) &=&\mathcal{B}\left( x\right) \pm \left(
y\cdot \nabla \right) \mathcal{B}\left( x\right) +\frac{1}{2}y^{\limfunc{tr}%
}\nabla ^{2}B\left( x\right) y+o\left( \left\vert y\right\vert ^{2}\right) ,
\\
\text{i.e. }\frac{\mathcal{B}\left( x+y\right) +\mathcal{B}\left( x-y\right) 
}{2} &=&\mathcal{B}\left( x\right) +\frac{1}{2}y^{\limfunc{tr}}\nabla
^{2}B\left( x\right) y+o\left( \left\vert y\right\vert ^{2}\right) ,
\end{eqnarray*}%
for sufficiently small $\left\vert y\right\vert $, and then the full force
of Theorem \ref{quant bil} shows that%
\begin{eqnarray*}
&&\frac{1}{2}y^{\limfunc{tr}}\nabla _{x}^{2}B\left( x\right) y+o\left(
\left\vert y\right\vert ^{2}\right) +2\left\vert y_{2}\right\vert x_{1}\leq
0, \\
\text{i.e. } &&2\left\vert y_{2}\right\vert x_{1}\leq C\left\vert
y\right\vert ^{2}\text{ for sufficiently small }\left\vert y\right\vert ,
\end{eqnarray*}%
which is clearly impossible since $x_{1}>0$. This shows that $\mathcal{B}%
\left( x\right) =\infty $ for some $x\in \Omega $ as we claimed just before
the statement of Lemma \ref{dyadic restricted}.

In order to achieve the doubling property in the third line of the
conclusion of Lemma \ref{dyadic restricted}, we follow \cite{NaVo} by fixing 
$0<\tau <1$ and modifying the above proof as follows.

\begin{enumerate}
\item Replace $\mathcal{F}_{J;x}$ with $\mathcal{F}_{J;x,\tau }$ where $%
\mathcal{F}_{J;x,\tau }$ consist of those pairs $\left( U,V\right) $ of
positive functions on $J$ such that%
\begin{eqnarray*}
E_{J}U &=&x_{1},\ \ \ E_{J}V=x_{2}, \\
\frac{\left\vert \triangle _{I}U\right\vert }{E_{I}U} &\leq &\frac{\tau }{10}%
,\ \ \ \text{for all dyadic }I\subset J, \\
\frac{\left\vert \triangle _{I}V\right\vert }{E_{I}V} &\leq &\frac{\tau }{10}%
,\ \ \ \text{for all dyadic }I\subset J, \\
\text{and }\left( E_{I}U\right) \left( E_{I}V\right) &<&1,\ \ \ \ \ \text{%
for all }I\in \mathcal{D}^{0}\text{ with }I\subset J.
\end{eqnarray*}

\item Replace $\mathcal{B}\left( x\right) $ with $\mathcal{B}_{\tau }\left(
x\right) $ where 
\begin{equation*}
\mathcal{B}_{\tau }\left( x\right) \equiv \sup_{J\in \mathcal{D}^{0}}\left\{ 
\frac{1}{\left\vert J\right\vert }\sum_{I\in \mathcal{D}^{0}:\ I\subset
J}\left\vert \bigtriangleup _{I}V\right\vert \left( E_{I}U\right) \left\vert
I\right\vert :\left( U,V\right) \in \mathcal{F}_{J;x,\tau }\right\} ,\ \ \ \
\ \text{for }x\in \Omega .
\end{equation*}
\end{enumerate}

Then we obtain from the above argument that $\mathcal{B}_{\tau }\left(
x\right) =\infty $ for some $x\in \Omega $. Indeed, the analogue of Theorem %
\ref{quant bil} is now the following theorem.

\begin{theorem}
\label{quant bil tau}Assume that $\mathcal{B}_{\tau }\left( x\right) <\infty 
$ for all\ $x\in \Omega $. If $y=\left( y_{1},y_{2}\right) $ is such that $%
x,x+y,x-y\in \Omega $ and $\frac{2\left\vert y_{1}\right\vert }{x_{1}},\frac{%
2\left\vert y_{2}\right\vert }{x_{2}}\leq \frac{\tau }{10}$, then%
\begin{equation*}
\frac{\mathcal{B}_{\tau }\left( x+y\right) +\mathcal{B}_{\tau }\left(
x-y\right) }{2}+2\left\vert y_{2}\right\vert x_{1}-\mathcal{B}_{\tau }\left(
x\right) \leq 0.
\end{equation*}
\end{theorem}

The point of assuming $\frac{2\left\vert y_{1}\right\vert }{x_{1}},\frac{%
2\left\vert y_{2}\right\vert }{x_{2}}\leq \frac{\tau }{10}$ in the
hypotheses of Theorem \ref{quant bil tau} is that the weight pair $\left( 
\widetilde{U},\widetilde{V}\right) $ constructed in the proof of Theorem \ref%
{quant bil} above then satisfies $\frac{\left\vert \triangle _{I}\widetilde{U%
}\right\vert }{E_{I}\widetilde{U}}=\frac{2\left\vert y_{1}\right\vert }{x_{1}%
}\leq \frac{\tau }{10}$ and $\frac{\left\vert \triangle _{I}\widetilde{U}%
\right\vert }{E_{I}\widetilde{U}}\leq \frac{\tau }{10}$ for $I\subsetneqq J$%
, and similarly for $\widetilde{V}$, and so we have $\left( \widetilde{U},%
\widetilde{V}\right) \in \mathcal{F}_{J;x,\tau }$. The proof of Theorem \ref%
{quant bil tau} now proceeds as in the proof of Theorem \ref{quant bil}. The
remainder of the argument is unchanged.

This completes the proof of Lemma \ref{dyadic restricted} since one easily
verifies that if $\frac{\left\vert \triangle _{I}U\right\vert }{E_{I}U}\leq 
\frac{\tau }{10}\ $for all dyadic $I\in \mathcal{D}^{0}$, then $1-\tau <%
\frac{E_{I_{-}}U}{E_{I_{+}}U}<1+\tau ,\ \ \ \ \ $for all $I\in \mathcal{D}%
^{0}$, and similarly for $V$.

\begin{remark}
The above argument proves that if $\Omega $ is a domain in $\mathbb{R}^{n}$,
and $B:\Omega \rightarrow \left[ 0,\infty \right] $ is twice differentiable
at some $x\in \Omega $, then we cannot have%
\begin{equation*}
B\left( x\right) \geq \frac{B\left( x+y\right) +B\left( x-y\right) }{2}%
+2\left\vert y_{2}\right\vert x_{1},\ \ \ \ \ \text{for all }y\text{ such
that }x\pm y\in \Omega .
\end{equation*}%
This simple observation doesn't apply to the Bellman function for testing
conditions in \cite[see (3.1)-(3.4)]{NaVo}, since in particular, the
inequality for the three dimensional Bellman function in \cite[(3.13)]{NaVo}
has $y_{2}^{2}x_{1}$ in place of $2\left\vert y_{2}\right\vert x_{1}$:%
\begin{equation*}
B\left( x\right) +x_{2}\frac{\partial B}{\partial x_{3}}y_{1}^{2}\geq \frac{%
B\left( x+y\right) +B\left( x-y\right) }{2}+y_{2}^{2}x_{1}\ .
\end{equation*}%
Moreover, the two problems are quite different, as the conclusion in \cite[%
see (4.1)-(4.3) plus doubling]{NaVo} yields a Muckenhoupt doubling weight
pair that satisfies one testing condition for the dyadic Hilbert transform,
but not the other; while Theorem \ref{beta} above yields a Muckenhoupt
doubling weight pair that cannot satisfy either testing condition, since
they each imply bilinear testing.
\end{remark}

\begin{problem}
Is the Bilinear Cube Testing constant $\mathcal{BCT}_{H}\left( \sigma
,\omega \right) $ for the Hilbert transform $H$ controlled by $\mathcal{A}%
_{2}^{\alpha }\left( \sigma ,\omega \right) $ when the measures $\sigma
,\omega $ are doubling?
\end{problem}

\section{Appendix}

Here we complete the analysis of \emph{energy nondegeneracy} conditions,
introduced in \cite{Saw2}, which arise when using Calder\'{o}n-Zygmund
decompositions in connection with weighted Alpert wavelets. We begin by
recalling some notation from \cite{Saw2}. We say that a polynomial $P\left(
y\right) =\sum_{0\leq \left\vert \beta \right\vert <\kappa }c_{\beta
}y^{\beta }$ of degree less than $\kappa $ is \emph{normalized} if 
\begin{equation*}
\sup_{y\in Q_{0}}\left\vert P\left( y\right) \right\vert =1,\ \ \ \ \ \text{%
where }Q_{0}\equiv \dprod\limits_{i=1}^{n}\left[ -\frac{1}{2},\frac{1}{2}%
\right) .
\end{equation*}

\begin{definition}
\label{def Q norm}Denote by $c_{Q}$ the center of the cube $Q$, and by $\ell
\left( Q\right) $ its side length, and for any polynomial $P$ set 
\begin{equation*}
P^{Q}\left( y\right) \equiv P\left( c_{Q}+\ell \left( Q\right) y\right) .
\end{equation*}%
We say that $P\left( x\right) $ is $Q$\emph{-normalized} if $P^{Q}$ is
normalized. Denote by $\left( \mathcal{P}_{\kappa }^{Q}\right) _{\limfunc{%
norm}}$ the set of $Q$-normalized polynomials of degree less than $\kappa $.
\end{definition}

Thus a $Q$-normalized polynomial has its supremum norm on $Q$ equal to $1$.
Recall from (\ref{def doub}) that a locally finite positive Borel measure $%
\mu $ on $\mathbb{R}^{n}$ is \emph{doubling} if there exist constants $%
0<\beta ,\gamma <1$ such that%
\begin{equation}
\left\vert \beta Q\right\vert _{\mu }\geq \gamma \left\vert Q\right\vert
_{\mu },\ \ \ \ \ \text{for all cubes }Q\text{ in }\mathbb{R}^{n}.
\label{doub}
\end{equation}%
Note that $\sup_{y\in Q}\left\vert P\left( y\right) \right\vert =\left\Vert 
\mathbf{1}_{Q}P\right\Vert _{L^{\infty }\left( \mu \right) }$ for any cube $%
Q $, polynomial $P$, and nontrivial doubling measure $\mu $. It was shown in 
\cite{Saw2} that if $\mu $ is doubling on $\mathbb{R}^{n}$, then for every $%
\kappa \in \mathbb{N}$ there exists a positive constant $C_{\kappa }$ such
that%
\begin{eqnarray}
\left\vert Q\right\vert _{\mu } &\leq &C_{\kappa }\int_{Q}\left\vert P\left(
x\right) \right\vert ^{2}d\mu \left( x\right) ,\ \ \ \ \ \text{for all cubes 
}Q\text{ in }\mathbb{R}^{n}\text{,}  \label{energy nondeg} \\
&&\text{and for all }Q\text{-normalized polynomials }P\text{ of degree less
than }\kappa .  \notag
\end{eqnarray}%
It was also shown that conversely, if $\kappa >2n$, then (\ref{energy nondeg}%
) implies that $\mu $ is doubling. Here we extend the converse to the
optimal range $\kappa >1$.

\begin{lemma}
\label{doubling}Let $\mu $ be a locally finite positive Borel measure on $%
\mathbb{R}^{n}$. If (\ref{energy nondeg}) holds for some positive integer $%
\kappa \in \mathbb{N}$, then $\mu $ is doubling.
\end{lemma}

\begin{proof}
Assume that (\ref{energy nondeg}) holds for some $\kappa \in \mathbb{N}$.
Momentarily fix a cube $Q$ and an index $1\leq i\leq n$, and let $a_{Q}\in 
\mathbb{R}^{n}$ where $Q=\dprod\limits_{i=1}^{n}\left[ \left( a_{Q}\right)
_{i},\left( a_{Q}\right) _{i}+\ell \left( Q\right) \right] $. Then the
polynomial 
\begin{equation*}
P_{i}\left( x\right) \equiv \frac{x_{i}-\left( a_{Q}\right) _{i}}{\ell
\left( Q\right) }
\end{equation*}%
is $Q$-normalized of degree less than $\kappa $, vanishes on the face of the
boundary of $Q$ which lies in the hyperplane $\left\{ x\in \mathbb{R}%
^{n}:x_{i}=\left( a_{Q}\right) _{i}\right\} $, and is $1$ on the opposite
face where $x_{i}=\left( a_{Q}\right) _{i}+\ell \left( Q\right) $. Thus for
each $0<\varepsilon <1$, there is $\beta <1$, sufficiently close to $1$, and 
\emph{independent} of the cube $Q$, so that%
\begin{eqnarray*}
\left\vert Q\right\vert _{\mu } &\leq &C_{\kappa }\int_{Q}\left\vert
P_{i}\right\vert ^{2}d\mu =C_{\kappa }\left\{ \int_{Q\cap \left\{ \frac{%
x_{i}-\left( a_{Q}\right) _{i}}{\ell \left( Q\right) }<1-\beta \right\}
}\left\vert P_{i}\right\vert ^{2}d\mu +\int_{Q\cap \left\{ \frac{%
x_{i}-\left( a_{Q}\right) _{i}}{\ell \left( Q\right) }\geq 1-\beta \right\}
}\left\vert P_{i}\right\vert ^{2}d\mu \right\} \\
&\leq &\varepsilon \left\vert Q\cap \left\{ \frac{x_{i}-\left( a_{Q}\right)
_{i}}{\ell \left( Q\right) }<1-\beta \right\} \right\vert _{\mu }+C_{\kappa
}\left\vert Q\cap \left\{ \frac{x_{i}-\left( a_{Q}\right) _{i}}{\ell \left(
Q\right) }\geq 1-\beta \right\} \right\vert _{\mu } \\
&\leq &\varepsilon \left\vert Q\right\vert _{\mu }+C_{\kappa }\left\vert
Q\cap \left\{ \frac{x_{i}-\left( a_{Q}\right) _{i}}{\ell \left( Q\right) }%
\geq 1-\beta \right\} \right\vert _{\mu }\ .
\end{eqnarray*}

Now we focus on the rectangle $Q\cap \left\{ \frac{x_{i}-\left( a_{Q}\right)
_{i}}{\ell \left( Q\right) }\geq 1-\beta \right\} $ that appears on the
right hand side above. It can be written as a union of at \ most $2^{n-1}$
cubes $Q^{\prime }\in \Gamma $ of side length $\beta \ell \left( Q\right) $
(thus not necessarily dyadic, and overlapping significantly - e.g. if $Q=%
\left[ 0,1\right] ^{2}$ and $i=1$, then the squares $Q^{\prime }$ are $\left[
1-\beta ,1\right] \times \left[ 0,\beta \right] $ and $\left[ 1-\beta ,1%
\right] \times \left[ 1-\beta ,1\right] $), where $\Gamma $ is an index set
of size $2^{n-1}$ satisfying%
\begin{equation*}
Q\cap \left\{ \frac{x_{i}-\left( a_{Q}\right) _{i}}{\ell \left( Q\right) }%
\geq 1-\beta \right\} =\dbigcup\limits_{Q^{\prime }\in \Gamma }Q^{\prime }.
\end{equation*}%
Now fix another index $j\neq i$, and for each of these cubes $Q^{\prime }$,
apply the above argument with the polynomial $P_{j}$ in place of $P_{i}$.
Then we obtain%
\begin{eqnarray*}
\left\vert Q\right\vert _{\mu } &\leq &\varepsilon \left\vert Q\right\vert
_{\mu }+C_{\kappa }\left\vert Q\cap \left\{ \frac{x_{i}-\left( a_{Q}\right)
_{i}}{\ell \left( Q\right) }\geq 1-\beta \right\} \right\vert _{\mu } \\
&\leq &\varepsilon \left\vert Q\right\vert _{\mu }+C_{\kappa
}\sum_{Q^{\prime }\in \Gamma }\left( \varepsilon \left\vert Q^{\prime
}\right\vert _{\mu }+C_{\kappa }\left\vert Q^{\prime }\cap \left\{ \frac{%
x_{j}-\left( a_{Q^{\prime }}\right) _{j}}{\ell \left( Q^{\prime }\right) }%
\geq 1-\beta \right\} \right\vert _{\mu }\right) \\
&\leq &\varepsilon \left( 1+2^{n-1}C_{\kappa }\right) \left\vert
Q\right\vert _{\mu }+2^{n-1}C_{\kappa }^{2}\left\vert Q\cap \left\{ \frac{%
x_{j}-\left( a_{Q}\right) _{j}}{\beta \ell \left( Q\right) }\geq 1-\beta 
\text{ and }\frac{x_{i}-\left( a_{Q}\right) _{i}}{\beta \ell \left( Q\right) 
}\geq 1-\beta \right\} \right\vert _{\mu },
\end{eqnarray*}%
where in the final term we have written $\ell \left( Q^{\prime }\right)
=\beta \ell \left( Q\right) $ and made the final set bigger by replacing $%
\left( a_{Q^{\prime }}\right) _{j}$ with the smaller number $\left(
a_{Q}\right) _{j}$. By further replacing the second factor of $2^{n-1}$ by
its square, we have%
\begin{equation*}
\left\vert Q\right\vert _{\mu }\leq \varepsilon \left( 1+\left[
2^{n-1}C_{\kappa }\right] \right) \left\vert Q\right\vert _{\mu }+\left[
2^{n-1}C_{\kappa }\right] ^{2}\left\vert Q\cap \left\{ \frac{x_{j}-\left(
a_{Q}\right) _{j}}{\ell \left( Q\right) }\geq \beta \left( 1-\beta \right) 
\text{ and }\frac{x_{i}-\left( a_{Q}\right) _{i}}{\ell \left( Q\right) }\geq
\beta \left( 1-\beta \right) \right\} \right\vert _{\mu }
\end{equation*}%
We continue this process until we have exhausted the indices in $\left\{
1,2,...,n\right\} $ and are left with cubes $Q^{\prime }$ that are at
distance at least $\beta ^{n-1}\left( 1-\beta \right) $ from each of the
hyperplanes $\left\{ x\in \mathbb{R}^{n}:x_{i}=\left( a_{Q}\right)
_{i}\right\} $ for $1\leq i\leq n$.

Next we turn our attention to the remaining $n$ faces of the boundary of $Q$%
, which lie in the hyperplanes $\left\{ x\in \mathbb{R}^{n}:x_{i}=\left(
a_{Q}\right) _{i}+\ell \left( Q\right) \right\} $ for $1\leq i\leq n$, using
the polynomials%
\begin{equation*}
\widehat{P}_{i}\left( x\right) \equiv \frac{\ell \left( Q\right) +\left(
a_{Q}\right) _{i}-x_{i}}{\ell \left( Q\right) }.
\end{equation*}%
We eventually obtain%
\begin{equation*}
\left\vert Q\right\vert _{\mu }\leq \varepsilon \left( 1+\left[
2^{n-1}C_{\kappa }\right] +...+\left[ 2^{n-1}C_{\kappa }\right]
^{2n-1}\right) \left\vert Q\right\vert _{\mu }+\left[ 2^{n-1}C_{\kappa }%
\right] ^{2n}\left\vert \beta ^{2n-2}\left( 1-\beta \right) ^{2}Q\right\vert
_{\mu }\ .
\end{equation*}%
Now we choose $\varepsilon =\frac{1}{2\left( 1+\left[ 2^{n-1}C_{\kappa }%
\right] +...+\left[ 2^{n-1}C_{\kappa }\right] ^{2n-1}\right) }$ to get 
\begin{equation*}
\left\vert Q\right\vert _{\mu }\leq 2\left[ 2^{n-1}C_{\kappa }\right]
^{2n}\left\vert \beta ^{2n-2}\left( 1-\beta \right) ^{2}Q\right\vert _{\mu
}\ ,
\end{equation*}%
which is (\ref{doub}) with $\gamma =\frac{1}{2C_{\kappa }^{2n}}$ and $\beta $
replaced by $\beta ^{2n-2}\left( 1-\beta \right) ^{2}$.
\end{proof}

\end{document}